\documentclass[12pt,final]{scrartcl}
\usepackage[english]{babel}
\usepackage{header_arxiv}
\usepackage{fixme}                     
\input xy
\xyoption{all}

\title{The Cleavage Operad and String Topology of Higher Dimension}
\author{Tarje Bargheer}
\begin{document}
\maketitle
\section{Introduction}\label{intro}

In \cite{UniversalAlgebra}, Voronov gave life to the cacti operad: An operad whose homology acts on the shifted homology of the free loop space over a compact, smooth, orientable manifold $M$ -- $\hh_*(M^{S^1})$ -- giving it the structure of a Batalin-Vilkovisky algebra; hereby recovering the Chas-Sullivan product of \cite{ChasSullivan}. As an intermediate operad, Kaufmann in \cite{SeveralCacti} gave an $E_2$-operad -- the spineless cacti operad -- whose homology acts to give a Gerstenhaber structure, underlying the Batalin-Vilkovisky structure of Chas and Sullivan's string topology; all reflected in the fact that appending the data of $\SO(2)$ to the spineless cacti suitably yields the cacti operad.

We follow the same general string of ideas, but generalize them by replacing $S^1$ with a manifold $N \subseteq \RR^{n+1}$ -- of arbitrary dimension -- embedded in euclidean space. Our methods will involve certain decompositions of $N$, and for convexly embedded spheres these decompositions are simple enough to obtain results within topology, our focus will thus take a shift towards $N:= S^n \subset \RR^{n+1}$ the unit sphere.

What we construct is a coloured operad that acts on $M^N$ -- the space of maps from $N$ to $M$ -- in a related manner to how the cacti operad acts on $M^{S^1}$. As revealed by the previous sentence, we found it necessary to broaden the scope of the use of the word 'operad' -- and enter the realm of coloured operads; coloured over topological spaces. As we describe in section \ref{operadgeneral}, this colouring is similar to picking a category internal to topological spaces, with traditional operads being one-object gadgets. We show in \ref{entheorem} that for $N = S^n$, the homotopy type of this operad is computable, using combinatorial methods of \cite{BergerCellular}, as a coloured $E_{n+1}$-operad. We then show how to form a semidirect product of this operad with $\SO(n+1)$, providing a $(n+1)$-Batalin-Vilkovisky structure on the homology of $M^{S^n}$ when $n$ is odd.

In \cite{StringTopologyNotes}[Ch. 5], an outline is given for a generalisation of the cacti operad to the $n$-dimensional cacti operad, by replacing lobes with copies of $S^n$ floating in $\RR^{n+1}$. Our original motivation was to explicitly compute the structure of this operad; attempting to construct homotopies that would provide equivalences between the little $(n+1)$-disks operad and the $n$-dimensional cacti operad, as was done in \cite{CultivatingTarje} for the 1-dimensional case. 
However, such attempts did not seem to have a shortcut, bypassing the structure of the diffeomorphism group $\Diff(S^n)$. 

Although we are working with coloured operads, and so give an operad different from the cacti operad, morally we take the stance of \cite{SeveralCacti} -- starting from a more rigid structure, where diffeomorphisms have no influence until the twisting of section \ref{defsym} can be inferred.
The coloured operad we define have an operadic structure that is basically given by cleaving $N$ into smaller submanifolds -- timber -- and therefore we dub the operad, the \emph{cleavage operad over $N$}, $\Cleav_N$.

\begin{figure}[h!tb]
\includegraphics[scale=0.6]{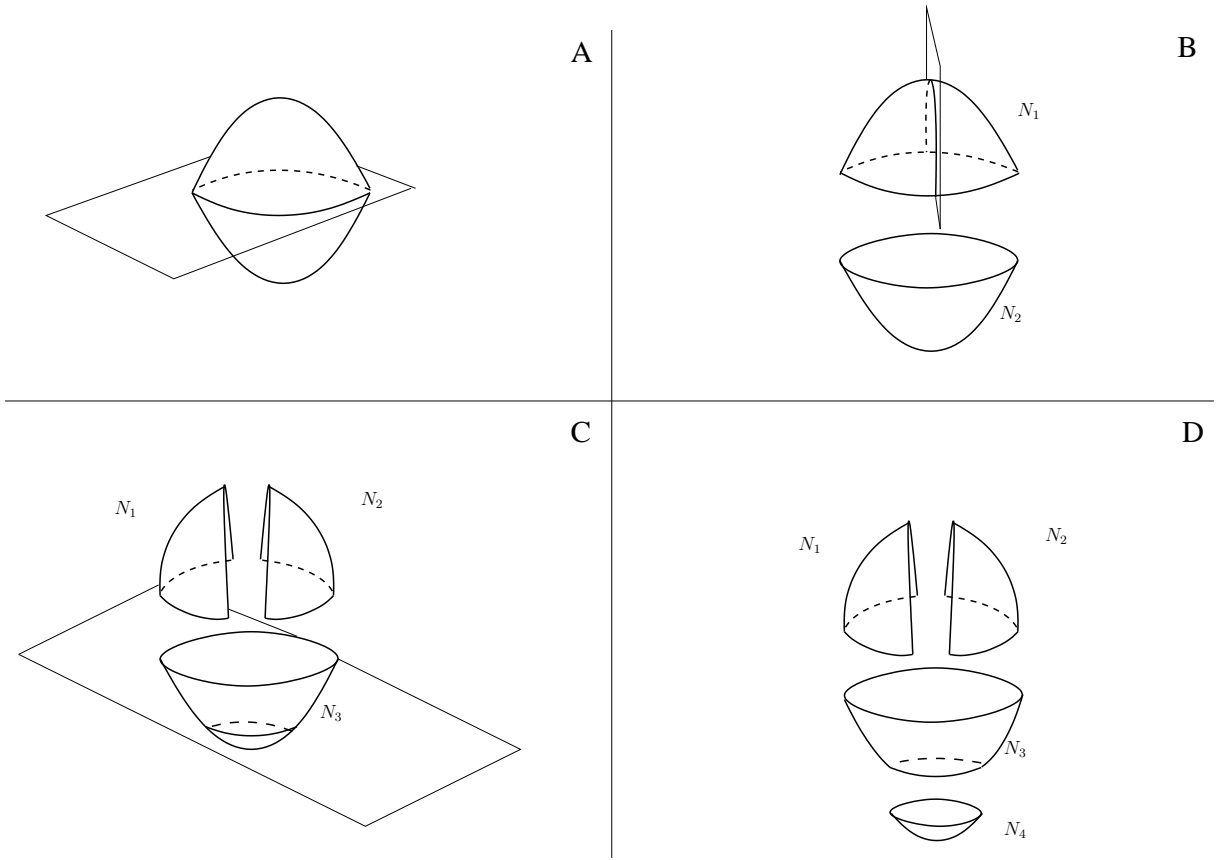}
\caption{Recursive procedure cleaving a sphere into four pieces of timber}
\label{cleavagepic}
\end{figure}

An example of such a cleavage is given in figure \ref{cleavagepic}, starting in picture A, where $N$ is cleaved by a single hyperplane chopping $N$ into two pieces of timber $N_1$ and $N_2$ in picture B; and hence succesively cleaving the timber produced into smaller subsets of $N$. Note that the non-linear ordering in which the cleaving is done constitutes part of the data; for instance, interchanging the cuts made in A and B would yield a different cleavage.

However, in order to have an interesting topology, we need to forget as much of the ordering dictated by the indexing trees as possible. In \ref{NCleave}, we forget what needs to be forgotten by defining $\Cleav_N$ as a quotient of an operad with $k$-ary structure given as pairs $(T,\underline{P})$ where $T$ is a tree of arity $k$ expressing the ordering with which $N$ should be cleaved, and $\underline{P} = \{P_1,\ldots,P_{k-1}\}$ a tuple of affine hyperplanes -- conveying information on where to cleave by decorating the internal vertices of $T$.

In terms of figure \ref{cleavagepic}, this quotient would allow for an interchange of the ordering of hyperplanes, such that the hyperplane of C cleaves before A and B, as well as allowing an interchange of the hyperplanes of B and C.

\subsection*{Summary of Results}
As usual in string topology, fix a dimension $d$ of the manifold $M$, and let $\hh_*(-) := \h_{*+d}(-)$.

Recall that the correspondence category $\Corr(\mathcal{C})$ over a co-complete category $\mathcal{C}$ is given by letting $\Ob(\Corr(\mathcal{C})) = \Ob(\mathcal{C})$, and the set of morphisms from objects $X$ to $Y$, $\Corr(\mathcal{C})(X,Y)$ be given as isomorphism classes of diagrams $\xymatrix{X & Z\ar[r]\ar[l] & Y}$ where $Z$ is an arbitrary object, and the arrows are morphisms in $\mathcal{C}$. Composition of correspondences given by forming pull-backs is associative by the universal properties of the pull-back.

\begin{theorema}\label{actionintro}
The operad $\Cleav_N$ acts on $M^N$ in the category of correspondences over topological spaces. 

\end{theorema}

We also indicate how this gives rise to the result that letting $N := S^n \subseteq \RR^{n+1}$ the unit sphere, and $M$ a compact, smooth, oriented $d$-manifold, $\h_*(\Cleav_{S^n})$ acts on $\hh_*(M^{S^n})$.



To examplify the action on correspondences, made precise in section \ref{cactiaction}, take the cleavage in picture D of figure \ref{cleavagepic}. Note first that covering $N$ by the submanifolds $N_1,N_2,N_3$ and $N_4$, the space of all pairwise intersections of the covering inside $N$ is a space with two components $C_1$ and $C_2$. Consider the space 
$$M^N_{C_1,C_2} := \{f \in M^N \mid f(C_1) = \{k_1\} \subseteq M, f(C_2) = \{k_2\} \subseteq M\}$$ -- i.e. the space of maps from $N$ to $M$ that are constant on all of $C_1$ and all of $C_2$. The correspondence $\xymatrix{M^N & M^N_{C_1,C_2}\ar[r]^{\phi}\ar[l]_{\iota} & \left(M^N\right)^4}$ is given with $\iota$ the inclusion. The other map $\phi$ maps into the $i$'th factor of $\left(M^N\right)^4$ by taking $\phi(f)(m) = f|_{N_i}(m)$ for $m \in N_i$ and for $m \notin N_i$ letting $\phi(f)(m)$ be the same constant as $f(C_i)$ where $C_i$ is the component separating $m$ from $N_i$.

In \cite{HomotopyString}, it was discovered that for the cacti operad, a passage to homology can be done via spectra to obtain so-called umkehr maps, which on homology provides a reversal of one the arrows in the correspondences. In \cite{BargheerPuncture}, the action is realized by taking homotopy groups of a stable map between spectra. 
In that paper, a further punctured version of the cleavage operad is introduced as technical assistance for the new ideas required to produce such an action. 
 This present paper will only indicate how the action works pointwise in $\Cleav_{S^n}$, as well as the $2$-ary term of the operad -- providing reason to the specifics of the definitions of $\Cleav_{S^n}$.



Section \ref{encomputation} and beyond are used to show the following:

\begin{theoremb}\label{enintro}
The coloured operad $\Cleav_{S^n}$ is a coloured $E_{n+1}$-operad. Taking a semidirect product of $\Cleav_{S^n}$ by $\SO(n+1)$ provides a coloured operad $\Cleav_{S^n} \rtimes {\SO(n+1)}$ whose homology coincides with the framed little $(n+1)$-disks operad. 
\end{theoremb}

Theorem \cite[5.4]{WahlBat} gives a statement on how these actions of $\Cleav_{S^n} \rtimes {\SO(n+1)}$ relates to higher Batalin-Vilkovisky algebras. 

The main technicality in proving the above theorem is the first statement of a coloured $E_{n+1}$-operad. We apply combinatorial methods of \cite{BergerCellular} to show this theorem in section \ref{encomputation}. The final part of the statement follows in section \ref{defsym} by the construction of semidirect products as given in \cite{WahlBat}. Briefly, since $\Cleav_{S^n}$ and its action on $M^{S^n}$ is well-behaved with respect to an action of $\SO(n+1)$, induced from its action on $\RR^{n+1}$, we can apply a coloured construction of a semidirect product. 

Theorem B improves on several results in the litterature; compared to \cite[prop. 22]{GruherSalvatore}, we give an $E_{n+1}$-structure instead of an $E_n$-structure. The geometry we introduce can also be seen as an expansion of the ring structure on higher loop spaces given in \cite[Th. 1.1 + Cor. 1.2]{KallelSalvatoreRational}. Furthermore, the geometric nature of our construction allows for the higher version of Batalin-Vilkovisky algebras that is not part of the statement in \cite{PoHuMapping}. 




\subsection*{Acknowledgement}

In producing the material for this paper, constituting part of my PhD-thesis, I wish to first and foremost thank my PhD-supervisor Nathalie Wahl for her incredible help in getting this work done, as well as Craig Westerland who has hosted me for large parts of my PhD program. Several other mathematicians have pushed me onwards in my thinking, and I wish to thank Joan Mill\'es, S\o ren Galatius and Alexander Berglund. Furthermore, several useful comments were given by my committee in the evaluation of my PhD-thesis, and these have been incorporated into my thesis; I therefore also want to thank Paolo Salvatore, Ulrike Tillmann and Ib Madsen.

The material of the section was conceived at the Centre for Symmetry and Deformation at University of Copenhagen. I wish to thank the Danish National Research Foundation for providing the financial support, making this incredible place to do mathematics a possibility.

A substantial revision of the work was carried out at my time as a Leibniz Fellow at the Mathematical Research Institute of Oberwolfach. I wish to thank the Institute for providing me with this wonderful opportunity for extreme focus.


\section{Operadic and Categorical Concepts}\label{operadgeneral}
In this section we shall introduce language from the world of operads, used troughout. For a general overview, we refer to the category theoretic \cite[2.1]{LeinsterHigher}, where they are called \emph{multicategories}.

\begin{definition}\label{colouroperaddefine}
A \emph{coloured operad} $\mathcal{C}$ consist of 
\begin{itemize}
\item A class $\Ob(\mathcal{C})$ of \emph{objects} or \emph{colours}.
\item For each $k \in \NN$ and $a,a_1,\ldots,a_k \in \Ob(\mathcal{C})$, a class of $k$-ary morphisms denoted $\mathcal{C}\left(a;a_1,\ldots,a_k\right)$. 
\item To $i \in \{1,\ldots,k\}$, $f \in \mathcal{C}(a;a_1,\ldots,a_k)$ an $k$-ary morphism and $g \in \mathcal{C}(a_i;b_1,\ldots,b_n)$ an $n$-ary morphism, we require an operadic composition of arity $k+n-1$, $f \circ_i g \in \mathcal{C}(a;a_1,\ldots,a_{i-1},b_1,\ldots,b_n,a_{i+1},\ldots,a_k)$.
\item Units $\One_a \in \mathcal{C}(a;a)$ for any object $a$.
\item An action of $\Sigma_k$. I.e., given $\sigma \in \Sigma_k$ a map $\sigma. \colon \mathcal{C}(a;a_1,\ldots,a_k) \to \mathcal{C}(a;a_{\sigma(1)},\ldots,a_{\sigma(k)})$ for all $k \in \NN$ and $a_i \in \Ob(\mathcal{C})$.
\end{itemize}

These are subject to the following conditions, where we to a $H \subseteq \{1,\ldots, m\}$, denote by $\Sigma_H$ the permutation group of the elements of $H$. As by convention let $\Sigma_{|H|}$ denote the permutation group on the first $|H|$ natural numbers. The unique monotone map $H \to \{1,\ldots,|H|\}$ defines an isomorphism $\rho_H \colon \Sigma_H \to \Sigma_{|H|}$:

\begin{itemize}
\item Associativity: For $f \in \mathcal{C}\left(a;a_1,\ldots,a_k\right), g \in \mathcal{C}\left(a_i;b_1,\ldots,b_m\right), h \in \mathcal{C}\left(b_j;c_1,\ldots,c_l\right)$ the identity $$f \circ_i \left( g \circ_j h\right)  = \left(f \circ_i g \right) \circ_{j+i -1} h$$
holds for $i \in \{1,\ldots,k\}$ and $j \in \{1,\ldots,m\}$. For further $s \in \mathcal{C}(a_r;d_1,\ldots,d_u)$ where $i < r$ and $u \in \NN$ we require that $$(f \circ_i g) \circ_r s = (f \circ_r s ) \circ_{i+u-1} g.$$
\item $\Sigma_k$-equivariance: For $\sigma \in \Sigma_{k+m}$ and $f \in \mathcal{C}(a;a_1,\ldots,a_k), g \in \mathcal{C}(a_i;b_1,\ldots,b_m)$ the identity
$$
\sigma.(f \circ_i g) = \sigma|_{I}.f \circ_{\sigma|_{I}(i)} \sigma|_J.g
$$
holds where $I := \{1,\ldots,i,i+m+1,\ldots,k+m\}$ is the set of integers from $1$ to $k+m$ excluding the set $J:= \{i+1,\ldots,i+m\}$. For $H \subseteq \{1,\ldots,k+m\}$ the permutation $\sigma|_H \in \Sigma_{|H|}$ is given, using $\rho_H$ from the top of this definition, as $\sigma|_H := \rho_H\left(\widetilde{\sigma|_H}\right)$ where in turn $\widetilde{\sigma|_H} \in \Sigma|_H$ is defined from $\sigma$ by requiring that to $r,p \in H$ we have $\widetilde{\sigma|_H}(r) < \widetilde{\sigma|_H}(p)$ whenever $\sigma(r) < \sigma(p)$ so $\sigma|_H$ permutes the ordered symbols of $H$ in the same way that $\sigma$ permutes $\{1,\ldots,|H|\}$.
\item Unit-identity: For $f \in \mathcal{C}(a;a_1,\ldots,a_k)$ we have 
$$
f \circ_i \One_{a_i} = f \textrm{ and } \One_{a} \circ_1 f = f
$$
for all $i \in \{1,\ldots,k\}$.
\end{itemize}


\end{definition}

Indeed, a classical \emph{operad} is simply a coloured operad with a single object. We shall refer to such gadgets as \emph{monochrome operads}. On the other hand, a category $\mathcal{C}$ is the same as a coloured operad with $\mathcal{C}(a;a_1,\ldots,a_k) = \emptyset$ for $k > 1$.

Familiar concepts like functors, hom-sets and adjoints are extended in the obvious ways to this multi-arity setting. 








\begin{definition}
Let $(\mathcal{A},\boxtimes)$ be a symmetric monoidal category. The \emph{underlying coloured operad} $\Und_\mathcal{A}$ is given by letting $\Ob(\Und_\mathcal{A}) = \Ob(\mathcal{A})$ and
$$\mathcal{\Und_{\mathcal{A}}}(a;a_1,\ldots,a_n) := \Hom_\mathcal{A}(a_1 \boxtimes \cdots \boxtimes a_n, a).$$
The usual (monochrome) endomorphism operad $\End_A$ of an object $A \in \mathcal{A}$ is given by considering the full subcategory of $\Und_{\mathcal{A}}$ generated by $\{A\} \subseteq \Ob(\Und_{\mathcal{A}})$.
\end{definition}
\begin{definition}\label{colouraction}
An \emph{action} of a coloured operad $\mathcal{C}$ on $\mathcal{A}$ is a functor $\alpha \colon \mathcal{C} \to \Und_\mathcal{A}$. A \emph{monochrome action} of $\mathcal{C}$ is a functor $\alpha \colon \mathcal{C} \to \End_{A}$ for an object $A \in \mathcal{A}$.
\end{definition}

In (string) topology, operads are sought after for their actions on topological entities. As stated in the introduction, we venture on the same basic safari, but seek monochrome actions of coloured operads. As long as we seek monochrome actions, the extra colours on the operad become somewhat opaque -- and we get actions similar to that of monochrome operads. To have topological actions of course requires topology to enter the game:

Let $\mathpzc{O}$ be a coloured operad and let $\mathpzc{O}(-;k)$ be the set of all $k$-ary morphisms of $\mathpzc{O}$. 

We shall use the convention that an element of $\mathpzc{O}(-;k)$ has $1$ incoming and $k$ outgoing colours attached to it; this convention suits the cleaving process of our main coloured operad, along with its action. The '$-$' is used as a placeholder for the incoming colour.


\begin{definition}\label{topologicaloperad}
Let $\mathpzc{O}$ be a coloured operad. We say that $\mathpzc{O}$ is a \emph{coloured topological operad} if both $\Ob(\mathpzc{O})$ and $\mathpzc{O}(-;k)$ are topological spaces, along with the data of the following commutative diagram in the category of topological spaces, involving a pullback for $m,k \in \NN$ and $i \in \{1,\ldots,k\}$:
$$
\xymatrix{
\mathpzc{O}(-;k+m-1) & \mathpzc{O}(-;k) \times_{\Ob(\mathpzc{O})} \mathpzc{O}(-;m)\ar[r]\ar[d]\ar[l]_{\circ_i} & \mathpzc{O}(-;m)\ar[d]|{\ev_{\inc}}
\\
& \mathpzc{O}(-;k)\ar[r]^{\ev_{i}}\ar[r] & \Ob(\mathpzc{O}),
}
$$
where $ev_i$ evaluates at the $i$'th outgoing colour and $\ev_{\inc}$ evaluates at the incoming colour. 

The action map $\alpha_\sigma \colon \mathpzc{O}(-;k) \to \mathpzc{O}(-;k)$ permuting the outcoming colours should also be continous, and this topological structure should naturally adhere to the associativity, unit and $\Sigma_k$-equivariance conditions as specified in \ref{colouroperaddefine}. 
\end{definition}

We shall be interested in taking homology of coloured topological operads. Any coefficient ring will do, so we shall suppress coefficients and use the notation $\h_*(-)$.

Note that homology does not in general preserve direct limits, so applying homology functors to the diagram in \ref{topologicaloperad} will not yield another pushout diagram. One can however define the homology of a coloured topological operad as the coloured operad defined partially by the induced diagram

$$
\xymatrix@C=1em{
\h_*\left(\mathpzc{O}(-;k+m-1)\right)\\
\h_*\left(\mathpzc{O}(-;k) \times_{\Ob(\mathpzc{O})} \mathpzc{O}(-;m)\right) \ar@/^/[drr]\ar@/_/[ddr]\ar@{-->}[dr]\ar[u]_-{\circ_i}\\
 & \h_*(\mathpzc{O}(-;k)) \Box_{\h_*(\Ob(\mathpzc{O}))} \h_*(\mathpzc{O}(-;m))\ar[r]\ar[d] & \h_*(\mathpzc{O}(-;m))\ar[d]|{(\ev_{\inc})_*}\\
& \h_*\left(\mathpzc{O}(-;k)\right)\ar[r]^{(\ev_{i})_*}\ar[r] & \h_*\left(\Ob(\mathpzc{O})\right),
}
$$
where $A \Box_C B$ denotes the pullback of maps $A \to C$ and $B \to C$ in graded modules over the coefficient ring.

This diagram can in turn be taken to lead to a notion of partially defined coloured operads. Partial in the sense that the dotted arrow to the pullback-space is not always invertible. For the purpose of actions of operads, such a slightly more technical notion of partially defined operads would generally suffice.

However, the operads we consider in this paper will all have contractible colours, and we can instead of introducing partial operads use the following proposition to see that in our case, applying the homology functor to our operads will result in classical operads. 


\begin{proposition}\label{colourtomono}
Assume that $\mathpzc{O}$ is a coloured topological operad with $\Ob(\mathpzc{O}) \simeq *$, and with evalution maps $\ev_i$ fibrations for all $i \in \{1,\ldots,k\}$, or with $\ev_{\inc}$ a fibration. Then applying homology to $\mathpzc{O}$ defines $\h_*(\mathpzc{O})$ as a classical monochrome operad in the category of graded modules over the coefficient ring. 

\end{proposition}
\begin{proof}
Since $\Ob(\mathpzc{O}) \simeq *$, and the evaluation maps are fibrations, the long exact sequence of homotopy groups along with the $5$-lemma tells us that the pullback spaces $\mathpzc{O}(-;k) \times_{\Ob(\mathpzc{O})} \mathpzc{O}(-;m)$ and $\mathpzc{O}(-;k) \times \mathpzc{O}(-;m)$ are homotopy equivalent for all $k,m \in \NN$ so we have that 
$$\h_*(\mathpzc{O}(-;k) \times_{\Ob(\mathpzc{O})} \mathpzc{O}(-;m)) \cong \h_*(\mathpzc{O}(-;k) \times \mathpzc{O}(-;m)).$$ Applying the Eilenberg-Zilber map gives an induced map $$\circ_i \colon \h_*(\mathpzc{O}(-;k)) \otimes \h_*(\mathpzc{O}(-;m)) \to \h_*(\mathpzc{O}(-;k+m-1))$$ used to define classical operads. By definition of coloured topological operads, this map satisfies the needed associativity, unity and $\Sigma_k$-invariance conditions. 
\end{proof}

\begin{definition}
A \emph{morphism} $F \colon \mathpzc{O} \to \mathpzc{P}$ of coloured topological operads is given by continous maps 
\begin{eqnarray}\label{naturaloperad}
F_{\Ob} \colon \Ob(\mathpzc{O}) \to \Ob(\mathpzc{P}), F_k \colon \mathpzc{O}(-;k) \to \mathpzc{P}(-;k)
\end{eqnarray}
for all $k,m \in \NN$. These maps should provide a natural transformation of the diagrams specified in \ref{topologicaloperad}, defining $\mathpzc{O}$ and $\mathpzc{P}$ as coloured topological operads.

A \emph{weak equivalence} of coloured topological operads is given by a zig-zag of morphisms, where all continous maps of (\ref{naturaloperad}) are weak homotopy equivalences. 
\end{definition}

In particular, a weak equivalence $\mathpzc{P} \simeq \mathpzc{O}$ of topological coloured operads induces an isomorphism $\h_*(\mathpzc{P}) \cong \h_*(\mathpzc{O})$.



\begin{example}\label{colouren}
Let $\mathpzc{P}$ be a monochrome operad, and $X$ a topological space. We form the trivial $X$-coloured operad over $\mathpzc{P}$, $\mathpzc{P} \times X$, by setting
\begin{itemize}
\item $\Ob(\mathpzc{P} \times X) := X$
\item $\mathpzc{P} \times X(-;k) := \mathpzc{P}(k) \times X$
\end{itemize}
evaluation maps $\mathpzc{P}(k) \times X \to X$ are given by the projection map, $\circ_i$-composition, pointwise in $X$, the same as $\circ_i$-composition in $\mathpzc{P}$.
\end{example}

We say that $X$ is a monochrome $E_n$-operad if $|\Ob(X)| = 1$ and it is an $E_n$-operad in the usual sense; that is, it is weakly equivalent to the little $n$-disks operad.

\begin{definition}\label{enoperaddefine}
We say that a coloured operad $\mathpzc{O}$ is a \emph{coloured $E_n$-operad} if there is a weak equivalences of operads between $\mathpzc{O}$ and $\mathpzc{P} \times \Ob(\mathpzc{O})$, where $\mathpzc{P}$ is a monochrome $E_n$-operad.
\end{definition}

In \ref{cellulite}, we use methods of \cite{BergerCellular} to give a combinatorial way of detecting a coloured $E_n$-operad. We then use this to show that the operad we define in the next section is a coloured $E_{n+1}$-operad.

\section{The Cleavage Operad}\label{cleavageoperadsection}

\subsection{Definition of the Cleavage Operad}
The operadic structure we shall define will be induced from the operadic structure of trees. The trees we consider will all, without further specification, be: 
\begin{itemize}
\item Finitely generated from $2$-ary trees, in the sense that all vertices are univalent or trivalent, and there are only finitely many vertices. 
\item Rooted, in the sense that there is a distinguished univalent vertex called the root. 
\item Labelled, in the sense that for a $k$-ary tree, the $k$ remaining univalent vertices are numbered from $1,\ldots,k$.
\item Planar, specifying edges out of a trivalent vertex as left- right- or down-going.
\end{itemize}

$\Tree$ and $\Tree(k)$ denotes the set of isomorphism classes of trees, respectively $k$-ary trees. Grafting of trees defines $\Tree$ as a (monochrome) operad.

Let $\Gr_n(\RR^{n+1})$ be the Grassmanian of oriented subvectorspaces of codimension one inside $\RR^{n+1}$. 

\begin{definition}\label{hyperplanedef} 
Let the space of \emph{
 affine, oriented hyperplanes} be given as
$$\Hyp^{n+1} := \Gr_n(\RR^{n+1}) \times \RR,$$
where a pair $(H,p) \in \Hyp^{n+1}$ defines an affine hyperplane $P \subset \RR^{n+1}$ by translating $H$ along $p\cdot \nu_H$, where $\nu_H$ is the unit vector specifying the orientation of $H$. The normal vector $\nu_H$ also specifies an orientation for the affine hyperplane defined by $(H,p)$.  






\end{definition}

\begin{definition}
Let
$$
\Tree_{\Hyp^{n+1}}(k) := \Tree(k) \times \left(\Hyp^{n+1}\right)^{k-1}.
$$
We call $(T,\underline{P}) \in \Tree_{\Hyp^{n+1}}(k)$ a \emph{$(n+1)$-decorated tree} of arity $k$. Here $\underline{P} = (P_1,\ldots,P_{k-1})$ denotes the tuple of elements of $\Hyp^{n+1}$. We shall specify the trivalent vertices of $T$ as $v_1,\ldots,v_{k-1}$, such that these are matched to the hyperplanes -- and we encourage the reader to pretend that $P_i$ is dangling from $v_i$.

Denote by $\Tree_{\Hyp^{n+1}}$ the operad constituted from the pieces above. 

\end{definition}

\begin{convention}\label{Nconvention}
Throughout this text, we denote by $N \subseteq \RR^{n+1}$ an embedded, smooth manifold. 

\end{convention}




For $P \in \Hyp^{n+1}$, $\RR^{n+1} \setminus P$ consist of the two components $\left(\RR^{n+1}\right)^P_+$  and $\left(\RR^{n+1}\right)^P_-$, where $\left(\RR^{n+1}\right)^P_+$ is the space in the direction of the normal-vector of $P$. We say that $P$ \emph{bisects} $\RR^{n+1}$ into these two open subsets of $\RR^{n+1}$. 




Let $(T,\underline{P}) \in \Tree_{Hyp^{n+1}}(k)$, and designate by $V_T$ the set of vertices of $T$ that are not the root. To our given manifold $N$ and an open submanifold $U \subseteq N$, we associate for each internal vertex $v \in V_T$ a subspace $U_v \subseteq N$:

If $v$ is the vertex attached through only a single edge to the root, we let $U_v = U$. Since $T$ is binary, for $v \in V_T$  a trivalent vertex the left-going and right-going edge connect $v$ to $v_-$ and $v_+$, respectively. Let $P_v \in \Hyp^{n+1}$ be the decoration at $v$. We let $U_{v_-} = (U_v) \cap \left(\RR^{n+1}_-\right)^{P_v}$ and $U_{v_+} = U_v \cap \left(\RR^{n+1}_+\right)^{P_v}$. This determines $U_v$ for all $v \in V_T$.

This timbering process is illustrated in figure \ref{timberpicture} for the case $N = \RR^2$, and three hyperplanes inside $\RR^2$, it gives three different examples where trees are decorated by the hyperplanes in some way.

\begin{figure}[h!tb]
\includegraphics[scale=0.6]{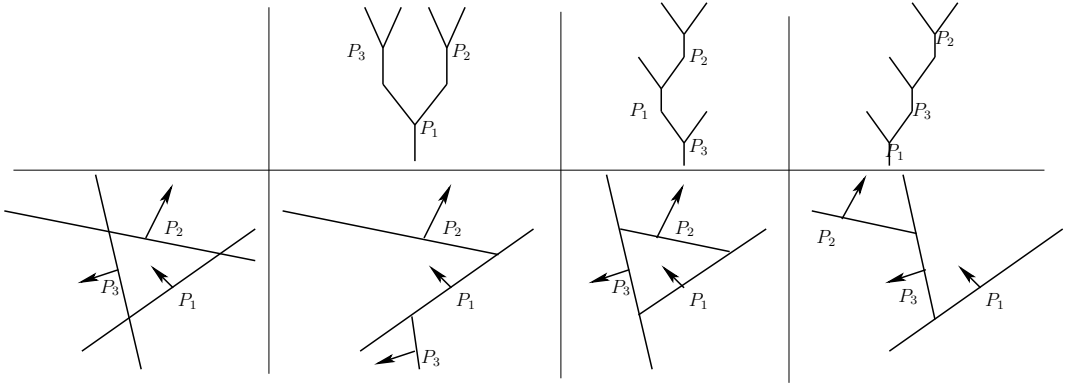}
\caption{The leftmost picture show three hyperplanes $P_1,P_2,P_3$ in $\RR^2$. In the three pictures to the right of it, it is shown how the hyperplanes enclose four subsets of $\RR^2$ that is the timber associated to decorating the hyperplanes on a tree. The tree that dictates the recursively defined cleaving process is shown directly above each picture of the timber. A right-going edge of a tree emanating from a vertex decorated by $P_i$ associates the subset of $\RR^2$ in the direction of the normal vector of $P_i$, and vice versa for a left-going edge and the subspace in the opposite direction of the normal vector of $P_i$.}
\label{timberpicture}
\end{figure}

\begin{definition}\label{cleaveconditions}
Let $U \subseteq N$ be an open submanifold. We say that a tree $(T,\underline{P}) \in \Tree_{Hyp^{n+1}}$ is \emph{$U$-cleaving} if to each trivalent $v \in V_T$ decorated by $P_v$ and associated with $U_v$ in the recursive process above have that $P_v$ intersects $U_{v}$ transversally, and that the hyperplane $P_v$ does not give rise to trivial timber, in the sense that both $U_{v_+} \neq \emptyset$ and $U_{v_-} \neq \emptyset$. 


We let $\Timber_N$ be the set of subsets of $N$, called \emph{timber}, where $U \in \Timber_N$ if there is an 
 $N$-cleaving tree, $T$, with $U$ associated to a leaf of $T$. 
\end{definition}

Hereby $\Timber_N$ consist of a subset of all open submanifolds of $N$, since at each vertex of a $N$-cleaving tree, the submanifolds associated to the two vertices above the vertex are again open submanifolds. Taking the closure inside $N$ of these submanifolds will yield a codimension 0 submanifold potentially with boundary and corners. 


There is a natural topology on $\Timber_N$ -- described by the space of hyperplanes giving rise to each timber, we assume this is given and wait until the next section with describing it explicitly to give the definition of the operad as fast as possible.

\begin{definition}\label{precleavage}
By the pre-$N$-cleavage operad, we shall understand the coloured operad $\widehat{\Cleav_N}$, given by
\begin{itemize}
\item $\Ob(\widehat{\Cleav_N}) = \Timber_N$
\item $\widehat{\Cleav_N}(U;k) := \left\{(T,\underline{P}) \in \Tree_{Hyp^{n+1}}(k) \mid (T,\underline{P}) \textrm{ is } U\textrm{--cleaving}\right\}$
\end{itemize}
Granted the topology on $\Timber_N$, we let 
$$\widehat{\Cleav_N}(-;k) = \coprod_{U \in \Ob(\widehat{\Cleav_N})} \widehat{\Cleav_N(U;k)}$$ and endow this with a topology as a subset of $\Tree_{\Hyp^n} \times \Timber_N$. The operadic composition


$$\circ_i \colon \widehat{\Cleav_N}(U;k) \times_{\Ob(\widehat{\Cleav_N})} \widehat{\Cleav_N}(-;m) \to \widehat{\Cleav_N}(U;k+m-1)$$ is given by grafting indexing trees, and retaining all decorations of the result. 

\end{definition}

\begin{definition}\label{NCleave}
We let the $N$-cleavage operad, $\Cleav_N$ be given by letting $$\Ob(\Cleav_N) = \{U \in \Timber_N \mid \complement U \simeq \coprod_{\textrm{finite}} *\}.$$ Here $\complement U$ denotes the complement of $U$ as a subspace of $N$, so the above space is given by restricting to the cases where the complement inside $N$ of timber has finitely many contractible components.

For the $k$-ary morphisms, we take the full suboperad of $\widehat{\Cleav_N}$ on the objects $\Ob(\Cleav_N)$ specified above, and apply a quotient: $\Cleav_N (-;k) := \widehat{\Cleav_N(-;k)}/\sim$, where $\sim$ is the equivalence relation given by letting $(T,\underline{P}) \sim (T',\underline{P}')$ if for all $i \in \{1,\ldots,k\}$ the $i$'th timber $N_i$ associated to $(T,\underline{P})$ is equal to the $i$'th timber $N_i'$ of $(T',\underline{P}')$. If $(T,\underline{P})$ and $(T',\underline{P}')$ are equivalent under $\sim$, we say they are \emph{chop-equivalent}.

Since the colours are left unchanged under $\sim$, taking operadic composition induced by $\widehat{\Cleav_N}$ is well-defined.

We denote an element of $\Cleav_N(-;k)$ by $[T,\underline{P}]$, where $(T,\underline{P})$ is a representative of the element.
\end{definition}

\begin{remark}\label{recursioncomment}

A priori it would suffice to be given the set $\Timber_N$ of subsets beforehand, and from this define the operadic structure through these subsets. In the sense that a $k$-ary operation of $\Cleav_N(U;k)$ corresponds to open subsets such that taking their individual closure defines a covering of $U$. This would avoid the introduction of trees and hyperplanes. 


However, in our forthcoming computations we shall see that it is important that we have this very strict relationship between the trees decorated by hyperplanes and the associated timber. If one had given a more arbitrary space of subsets of $N$ instead of $\Timber_N$, the same combinatorial benefits would not be available for computations.  
\end{remark}

\begin{remark}\label{schoenflies}
Note that for $N=S^n$, the complement inside $S^n$ of $U \in \Ob(\Cleav_{S^n})$ is by the generalized Sch\"onflies theorem \cite{BrownSchoenflies} always given by a disjoint union of wedges of disks -- the wedging occurring when hyperplanes intersect directly at $S^n$. 
\end{remark}

\begin{remark}\label{chipnotchop}
That we for $\Cleav_N$ have taken a subspace of objects; i.e. $\Ob(\Cleav_N) \subset \Timber_N$ is necessary in order to obtain homological actions via umkehr maps, as we shall see in the next section. Do however note that it is only necessary for the homological actions, meaning that this contractibility assumption could be skipped if one is only interested in actions through correspondences. We shall not invent notation for it, but in this sense theorem A as a more interesting statement over a less restrictive operad.

For $N = S^1$ we have $\Ob(\Cleav_{S^1}) = \Timber_{S^1}$; the complements $\complement U$ for $U \in \Ob(\Cleav_{S^1})$ are always intervals.

Letting $N = S^n$ and $n>1$,  it follows by a simple Mayer-Vietoris argument of the $\h_0$-groups along the closure of the timber $\overline{U}$ and $\complement U$ associated to $U \in \Ob(\Cleav_{S^n})$, that we in taking the subspace $\Ob(\Cleav_{S^n}) \subset \Timber_{S^n}$ are excluding the $U$ that are disconnected\footnote{such disconnected timber do exist; attempt for instance eating an apple conventionally to disconnect the peel in the last bite}. Using the generalized Sch\"oenflies theorem as stated in \ref{schoenflies} we get that all $[T,\underline{P}]$ with in- and out-put connected are indeed in $\Cleav_{S^n}$. 




We shall focus on the case of $S^n \subset \RR^{n+1}$ the unit sphere in this paper. We have however given a definition $\Cleav_N$ for generally embedded submanifolds $N \subseteq \RR^{n+1}$. The condition on the complements of timber having contractible components giving the inclusion $\Ob(\Cleav_N) \subseteq \Timber_N$ is very restrictive for manifolds different from $S^n$. The following is a list of observations regarding divergences from $N$ being the unit sphere in $\RR^{n+1}$.

\begin{itemize}
\item If we take an embedding $e \colon S^n \subseteq \RR^{n+1}$ that does not have a convex interior, it can be seen that a single cleaving hyperplane can produce an arbitrary even amount of timber. Again, the generalized Sch\"oenflies theorem tells us that these are all in $\Ob(\Cleav_{e(S^n)})$, allowing for homological actions of $M^{S^n}$ associated to these embeddings.
\item Taking a disjoint union of embedded spheres $e \colon \coprod^k S^n \to \RR^{n+1}$ where each individual embedding has convex interior, this again will lead to homological actions of $M^{\coprod^k S^n}$.
\end{itemize}

While the above two observations don't diverge far away from spheres, they can be considered as the leading examples producing an entangled version of string topology by twisting it via Khovanov homology. This will be accounted for in an upcoming paper. 

As for taking $N$ a manifold with components having different homeomorphism type than that of the sphere, we have the following comments: 

\begin{itemize}
\item As an example take $N=S^1\times S^1$. Forming the subset $\Ob(\Cleav_N) \subset \Timber_N$ makes $\Cleav_N(-;k) = \emptyset$ for $k > 1$. This can be seen by noting that if we assume that $[T,\underline{P}] \in \Cleav_N(-;2)$ then one of the outgoing timber will either be a cylinder or have a genus, hence one complement is non-contractible. For higher $k$, the complement of outgoing timber will contain such outgoing timber of $\Cleav_N(-;2)$ as a subset. And since they are subsets of $S^1 \times S^1$, there will always be at least one timber that is not contractible. 

We do not know of an example of a compact manifold without boundary embedded in euclidean space, where this phenomenon of vanishing higher arity operation does not occur; and brutally mimicking our cleaving methods does not appear to be the right way of producing homological operadic structure associated to $M^N$ for these types of manifolds $N$.
\item If we take the standard cylinder $N = S^n \times [-1,1] \subset \RR^{n+2}$, one can see that the above phenomenon of vanishing higher operations does not occur. In fact, it is easy to see that the homotopy type of $\Cleav_{S^n \times [-1,1]}$ is the same as that of $\Cleav_{S^n}$. The same occurs for the M\"oebius band in $\RR^3$ compared to the unit circle. We find it an interesting endeavour to figure out what happens with $\Cleav_E$ for general vector bundles $E \to S^n$, where $E$ is embedded in sufficiently large euclidean space. 
\end{itemize}
\end{remark}

The following two definitions give relations among trees in $\Cleav_N(-;k)$. These two types of relations -- it turns out -- are essential in the forthcoming material. One should however note that these two types of relations are only examples, and do not generate all relations imposed on $\Cleav_N(-;k)$. 

To $P \in \Hyp^{n+1}$, let $-P \in \Hyp^{n+1}$ denote the hyperplane given by reversing orientation of $P$.

\begin{definition}\label{Crelex}
Assume that $(T,\underline{P})$ is an $N$-cleaving tree. Let $v$ be an internal vertex of $T$, decorated by $P$. Let $T'$ be the tree obtained from $T$ by interchanging the branches above $v$, and let $\underline{P}'$ denote the set of hyperplanes with $P$ interchanged with $-P$.  
\begin{center}
\begin{figure}[h!tb]
\includegraphics[scale=0.7]{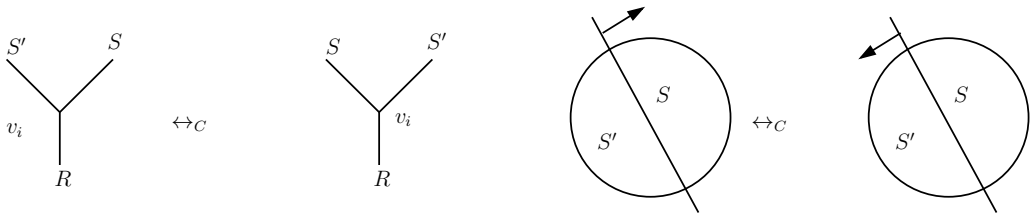}
\caption{The $\leftrightarrow_{C}$-relation}
\label{Crelatepic}
\end{figure}
\end{center}
Alternatively, the local picture \ref{Crelatepic} defines an equivalence relation $(T,\underline{P}) \leftrightarrow_C (T',\underline{P}')$. In $\Cleav_N(-;k)$, we have that $[T,\underline{P}] = [T',\underline{P}']$.
\end{definition}

We say that two hyperplanes $P,Q \in \Hyp^{n+1}$ are \emph{antipodally parallel} if $-Q$ can be obtained from $P$ by translating $P$ via its normal vector.

\begin{definition}\label{Brelex}
The local picture between $(T,\underline{P})$ and $(T',\underline{P})$ $N$-cleaving in \ref{Brelatepic}, describes when two $N$-cleaving trees are $\leftrightarrow_B$-related.
\begin{center}
\begin{figure}[h!tb]
\includegraphics[scale=0.7]{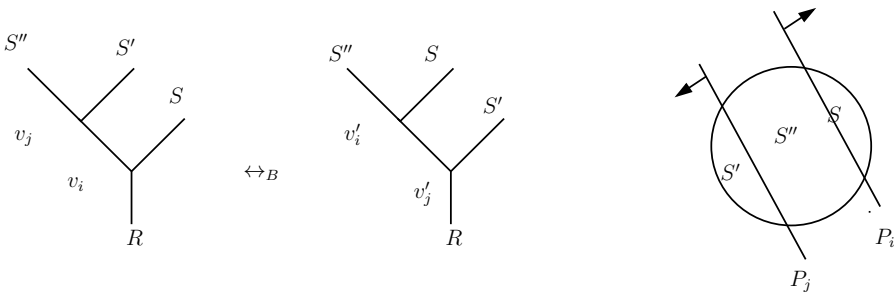}
\caption{The $\leftrightarrow_{B}$-relation}
\label{Brelatepic}
\end{figure}
\end{center}

Here the internal vertices $v_j$ directly above $v_i$ in $T$, decorated by $P_i$ and $P_j$ of $\underline{P}$ have swapped position -- along with the branches specified in the picture -- in $T'$ compared to $T$. We let $(T,\underline{P}) \leftrightarrow_B (T',\underline{P})$ if $P_i$ and $P_j$ are antipodally parallel.

In this case we have $[T,\underline{P}] = [T',\underline{P}]$ in $\Cleav_N(-;k)$.
\end{definition}

We say that $P$ and $Q$ are parallel if either $P,Q$ or $P,-Q$ are antipodally parallel.

\begin{observation}\label{paralb}
Assume that we are given $(T,\underline{P}) \in \Cleav_N(-;k)$, where all hyperplanes of $\underline{P}$ are pairwise parallel. Using the $\leftrightarrow_C$ and $\leftrightarrow_B$-relations of \ref{Crelex} and \ref{Brelex}, we obtain that $[T,\underline{P}] = [L_k,\underline{P}']$, where $L_k$ is a leftblown tree as in \ref{leftblown}, and $\underline{P}'$ is obtained from $\underline{P}$ by reversing the orientations along some hyperplanes.
\end{observation}

\subsection{Topology on the Timber}\label{topologytimber}

\begin{definition}\label{leftblown}
We let the \emph{arity $k$ left-blown tree} be the tree $L_k \in \Tree(k)$, with the right-going edges all ending at leaves, let the only leaf on a left-going edge be labelled by $k$ -- and for the other leaves, if there are $i$ internal vertices between the leaf and the root, we label the leaf by $i$.
\begin{figure}[h!tb]
\begin{center}
\includegraphics[scale=0.8]{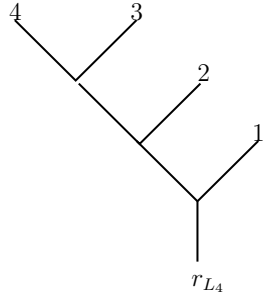}
\end{center}
\caption{The left-blown tree $L_4$}
\end{figure}

Let $L_1$ denote the tree with $V_{L_1} = \emptyset$, and a single leaf and root.
\end{definition}

Cultivating the cleaving tree appropriately, that is by reversing orientations of hyperplanes -- swapping branches around as in \ref{paralb} -- and cutting away unnecessary branches, we can assume $U \in \Timber_N$ to be on the leaf labelled $k+1$ of $L_k$ for some $k \in \NN$.    



\begin{construction}
We specify a topology on $\Timber_{\RR^{n+1}}$ as a subspace by seeing that there is an injection $\psi \colon \Timber_{\RR^{n+1}} \to \coprod_{i \in \NN} \left(\left(\Hyp^{n+1}\right)^i/\Sigma_i\right)$, where the permutation group $\Sigma_i$ permutes the factors of the product.

To specify the injection $\psi$, note that for $U \in \Timber_{\RR^{n+1}}$, the boundary of the closure in $N$ of $U$, $\partial \overline{U}$ contain the information needed to reconstruct $(L_k,\underline{P})$ having $U$ as the decoration on the top-leaf. Such hyperplanes are given by taking least affine subsets containing certain parts of $\partial \overline{U}$; either distinguished by different components of $\partial \overline{U}$ -- and otherwise a corner of $\partial \overline{U}$ will be the distinguishing feature for $(L_k,\underline{P})$. The function $\psi$ now maps $U\in \Timber_{\RR^{n+1}}$ to the corresponding hyperplanes, $(P_1,\ldots,P_k) \in \left(\Hyp^{n+1}\right)^k$ decorating $L_k$ and determined by $\partial \overline{U}$. This thus hits the component in the image of $\psi$ indexed by $k$.

There is ambiguity in the above definition of $\psi$; any reordering of the hyperplanes $(P_1,\ldots,P_k)$ will give rise to the same top-level timber. We therefore quotient by $\Sigma_i$ in the image of $\psi$. 
\end{construction}

Let $\Timber_N^\emptyset := \Timber_N \cup \{\emptyset\}$.

\begin{construction}
For $N \subset \RR^{n+1}$, we have a surjection $\mu_N \colon \Timber_{\RR^{n+1}} \to \Timber_N^{\emptyset}$, given by $\mu_N(U) = \overline{N \cap U}^\circ$, whenever $\overline{N \cap U}^{\circ} \in \Timber_N$, where the $\overline{\phantom{w}}$ followed by $^\circ$ denotes the opening of the closure inside $N$; and if $\overline{N \cap U}^\circ \notin \Timber_N$, we let $\mu_N(U) = \emptyset$. We specify a topology on $\Timber_N^{\emptyset}$ by letting $\mu_N$ be a quotient map.

Note that $\mu_N$ is well-defined since hyperplanes resulting in transversal intersections of $N$ gives rise to $\mu_N$ mapping to the empty set. We need to take the closure of the opening of $N \cap U$ to ensure that hyperplanes cleaving $N$ tangentially has the same effect as cleaving $\RR^{n+1}$ away from $N$. 

We let $\Timber_N \subset \Timber_N^{\emptyset}$ be endowed with the subspace topology
\end{construction}

\begin{remark}\label{killleaf}
Specifying $\mu_N$ as a quotient map means that certain elements $[L_k,\underline{P}] \in \Cleav_{\RR^{n+1}}$ will give rise to the same $\overline{U \cap N}^{\circ}$ at the top-leaf, under $\mu_N$. In particular, if $[L_k,\underline{P}]$ giving rise to $\overline{U \cap N}^\circ$ has some leaf decorated by $\emptyset$, we can instead consider $[L_{k-1},\hat{\underline{P}}]$ as giving rise to $\overline{U \cap N}^{\circ}$, where $\hat{\underline{P}}$ is given by removing the hyperplane from $\underline{P}$ decorating the vertex below said leaf, since the hyperplane in question will not be cleaving $N$.
\end{remark}

\begin{proposition}\label{contractibletimber}
Let $N$ be a compact submanifold of $\RR^{n+1}$. $\Timber_{N}$ is contractible. Similarly, $\Ob(\Cleav_{S^n})$ is contractible for $S^n \subseteq \RR^{n+1}$ the unit sphere
\end{proposition}
\begin{proof}

Given a point $U \in \Timber_N$, this will have $\complement U$ consist of a disjoint union of submanifolds of $N$ that has boundary at the points where $U$ has been cleaved from $N$ by hyperplanes $P_1,\ldots,P_k$. Each $P_i$ has a normal-vector in the direction towards $U$, and one in the direction away from $U$. The topology on $\Timber_N$, precisely determined by these hyperplanes makes it continuous in $\Timber_N$ to translate $P_1,\ldots,P_k$ in the direction away from $U$. Since $N$ is compact, this translation will in finite time take each hyperplane past tangential hyperplanes of $N$. Hence each hyperplane eventually disappears from the cleaving data and by \ref{killleaf}, eventually this translation provides an element of $\Cleav_{N}$ given by the $1$-ary undecorated tree $L_1$ as an operation from $N$ to $N$ . This hence defines a homotopy $\Phi_t \colon \Timber_N \to \Timber_N$ with $\Phi_0(U) = U$ and $\Phi_1(U) = N$, and hence the desired null-homotopy onto $N \in \Timber_N$.

For the statement on $\Ob(\Cleav_{S^n})$, note that from the definition \ref{NCleave} and \ref{schoenflies}, that the submanifolds of $\complement U$ will consist of a disjoint union of disks. The null-homotopy $\Phi_t$ above will in this case result in smaller and smaller disks as $t$ increases, and so $\Phi_t(U)$ remains within $\Ob(\Cleav_{S^n})$, and the null-homotopy is given as above.

 

\end{proof}

The following proposition tells us in conjunction with \ref{contractibletimber} that \ref{colourtomono} applies to $\Cleav_N$.

\begin{proposition}\label{evincfibration}
The evaluation map $\ev_{\inc} \colon \Cleav_{N}(-;k) \to \Ob(\Cleav_{N})$ is a fibration
\end{proposition}
\begin{proof}
To $[T,\underline{P}] \in \Cleav_{N}(-;k)$ we shall first of all for each of the hyperplanes $P_i$ of $\underline{P}$ prescribe the following transformation:

Under the relation \ref{Crelex}, we have to make a choice of normal-vector $\nu_i$ of $P_i$, this defines an interval $J_i = ]j_-,j_+[$ given by the maximal interval such that translating $P_i$ along $\nu_i$ with $r \in J_i$ as a scalar the hyperplane still participates in a cleaving configuration as a decoration of $T$. Note that since $[T,\underline{P}]$ is cleaving we have $0 \in J_i$ and denote by $c(J_i)$ the center-point of the interval. Note that the other choice of normalvector $-\nu_i$ will give rise to the interval $-J_i$ which will leave the following invariant:  

Fix $\varepsilon > 0$, if $j^i_{\min} := \min\{|j_-|,j_+\} < \varepsilon$, translate $P_i$ by $\sgn(c(J_i))\cdot\min\{\varepsilon - j^i_{\min},c(J_i)\}$ where $\sgn(c(J_i))$ is the sign of $c(J_i)$.

This translation can naturally be done to all the decorations of a decorated tree $[T,\underline{P}] \in \Cleav_{N}(-;k)$ simultaneously. Call this transformation $\Gamma_{\varepsilon}(T,\underline{P})$, note that since we are moving all hyperplanes at once, dependent on how large $\varepsilon$ is chosen, $\Gamma_{\varepsilon}([T,\underline{P}])$ does not a priori result in a cleaving tree.

We seek to find a lift in the diagram
$$
\xymatrix{
Y\ar[r]^(0.4){\phi}\ar[d] & \Cleav_{N}(-;k)\ar[d]|{\ev_{\inc}}\\
Y \times I \ar[r]^{h}\ar@{-->}[ur]^{\tilde{h}} & \Ob(\Cleav_{N})
}
$$
where we can assume that $Y$ is compact, and therefore pick $$0< \varepsilon < \inf_{y \in Y}(\min\{j^i_{\min} \mid P_i \textrm{ decorates } \phi(y)\})$$ where $j^i_{\min}$ is the minimal value where $P_i$ can be translated in order to have it still participate in a cleavage as defined above.

The lift $\tilde{h}(y,t)$ is now given as $\Gamma_{\varepsilon}(\phi(y))$ considered as a cleaving tree of the timber $h(y,t)$. Note that our choice of $\varepsilon$ makes $\Gamma_{\varepsilon}(\phi(y))$ result in an element of $\Cleav_N(-;k)$, basically since along $t$, the timber $h(y,t)$ will change continuously and therefore by definition of $\Gamma_{\varepsilon}$ will for a small neighborhood of $t \in I$ only give rise to a small change in how the configurations of hyperplanes change, guaranteeing their continual cleaving attributes.
\end{proof}

\section{Action of Cleavages}\label{cactiaction}
Let $M$ be a compact oriented manifold. We set $M^N := \{f \colon N \to M\}$ -- i.e. the space of unbased, continous maps from $N$ to $M$, endowed with the compact-open topology. 



From this section on, we shall only consider our main example of $\Cleav_N$, namely the case where $N := S^n$ denotes the unit sphere bounding the unit disk $D^{n+1} \subset \RR^{n+1}$. 

With slight modifications, one can also obtain correspondence actions of general embedded manifolds $N \subseteq \RR^{n+1}$, we refer to our thesis, \cite[section 1.4]{Bargheerthesis} for the details.

Above \ref{cleaveconditions}, we have specified a procedure that to an $N$-cleaving tree $(T,\underline{P})$ associates at each vertex $v$ of $T$ an open submanifold $U_v \subseteq S^n$. This procedure extends to the manifold $D^{n+1}$ as well, so that every vertex $v$ of $T$ has a subset $D^{n+1}_v$ associated to it, denoted $\Rec(U_{v})$  where the boundary of $D^{n+1}_v$ will be the space $U_v \subseteq S^n$. 

\begin{definition}\label{blueprintdefine}
Let $U \in \Timber_{S^n}$. To a $U$-cleaving tree $(T,\underline{P})$, we associate the \emph{blueprint} of $(T,\underline{P})$ as the following subset of $D^{n+1}$:
$$
\beta_{(T,\underline{P})} := D^{n+1} \setminus \bigcup_{i=1}^{k} D^{n+1}_i
$$

where $D^{n+1}_i$ is the subset of $D^{n+1}$ associated to the $i$'th leaf of $(T,\underline{P})$ seen as cleaving $D^{n+1}$. By definition of the recursive procedure above \ref{cleaveconditions}, $\beta_{(T,\underline{P})}$ will be contained in the collective union of all the hyperplanes of $\underline{P}$, loosely it will consist of all points of hyperplanes in $\underline{P}$ that have been involved in the recursive bisection process of $D^{n+1}$.
\end{definition}
In figure \ref{cleavagepic} of the introduction, the boundary of $\beta_{(T,\underline{P})}$ will be the collective boundary of the closure within $S^n$ of the submanifolds in picture D. 

The notion of a blueprint is invariant under the chop-equivalence of \ref{NCleave}, so we can make sense of the blueprint for $[T,\underline{P}] \in \Cleav_{S^n}(-;k)$ and shall denote this by $\beta_{[T,\underline{P}]}$. 





Let $\Corr(\mathcal{C})$ denote the correspondence category over $\mathcal{C}$ a co-complete category, as described in the introduction.

\begin{construction}\label{cleavageaction}
We construct a functor $\Phi_{S^n} \colon \Cleav_{S^n} \to \End^{\Corr(\Top)}_{M^{S^n}}$. That is, an action of $\Cleav_{S^n}$ on $M^{S^n}$ as an object of the category of natural transformations of correspondences over $\Top$.

To $[T,\underline{P}] \in \Cleav_{S^n}(U;k)$ let $\complement {S^n}_1,\ldots,\complement {S^n}_k$ denote the complement, inside of $S^n$, of the timbers associated to the leafs of $T$. 


The action is given through the following pullback-diagram

\begin{eqnarray}\label{basaction}
\xymatrix{
M_{[T,\underline{P}]}^{S^n} \ar[r]^{\phi^*}\ar[d] & \left(M^{S^n}\right)^k\ar[d]|{\res} \\
 M^{\pi_0(\beta_{[T,\underline{P}]})} \ar[r]^{\phi} & \prod_{i=1}^k M^{\left(\complement {S^n}_{i}\right)}
}
\end{eqnarray}


where $\res$ is the restriction map onto each complement.




Where as usual $\pi_0(X)$ denotes the set of path-components of the space $X$. 

We define $\phi$ as the induced of a map $c \colon \coprod_{i=1}^k \left(\complement {S^n}_{i}\right) \to \pi_0(\beta_{[T,\underline{P}]})$. To define $c$, note that a component $C$ of $\complement {S^n}_i$, will have $\partial C \setminus \partial \overline{U}$ be the result of some cuts of hyperplanes decorating $T$. By definition of $\pi_0(\beta_{[T,\underline{P}]})$, and since $C$ is a connected component, the cuts will all constitute the same element, $c(C)$ -- of $\pi_0(\beta_{[T,\underline{P}]})$ making $\phi$ well-defined as the map constant map along the timber intersecting elements of $\pi_0(\beta_{[T,\underline{P}]})$ nontrivially.




By glueing the functions in the pullback space, we can identify $M_{[T,\underline{P}]}^{S^n}$ as the space of $f \in M^{S^n}$ such that $f$ is constant along each subspace of the blueprint that is a representative of $\pi_0(\beta_{[T,\underline{P}]})$. We hence have a canonical inclusion $\iota_{[T,\underline{P}]} \colon M^{S^n}_{[T,\underline{P}]} \to M^{S^n}$, and in turn a correspondence 
$$
\xymatrix{
M^{S^n} & M^{{S^n}}_{[T,\underline{P}]}\ar[l]_{\iota_{[T,\underline{P}]}}\ar[r]^{\phi^*} & \left(M^{S^n}\right)^k
}
$$
\end{construction}

\begin{example}
\begin{figure}[h!tb]
\includegraphics[scale=0.5]{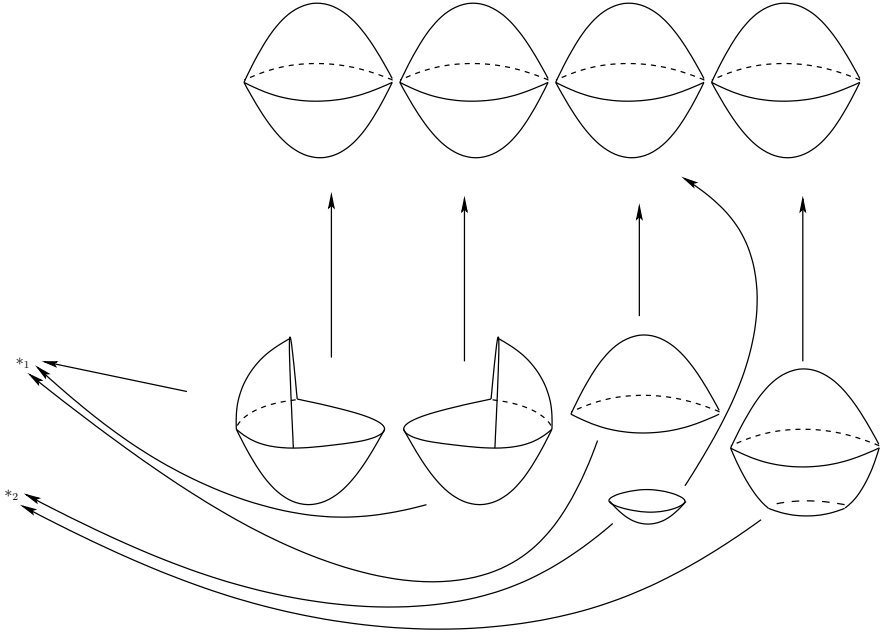}
\caption{The $4$-ary operation of the introduction has the complement of its timber $\complement {S^2}_1,\complement {S^2}_2,\complement {S^2}_3,\complement {S^2}_4$ drawn as the five disks on the bottom right corner. The upwards arrow are the inclusion maps, so that dualizing them provides the restriction map. The leftwards maps have as target two points, and these should be considered as collapsing the components of the blueprint $\beta_{[T,P_1,P_2,P_3]}$. These maps are given as the ones where the boundary of each disk is contained in a component of the blueprint.}
\label{actioncleave}
\end{figure}

For the cleavage $[T,P_1,P_2,P_3] \in \Cleav_{S^2}(S^2;4)$ in picture D of figure \ref{cleavagepic} in the introduction, we can consider the morphisms defining diagram (\ref{basaction}) through figure \ref{actioncleave}, using the mapping functor $M^{(-)}$ to dualize the morphisms indicated in the figure, we get the maps that define the pullback diagram (\ref{basaction}).
\end{example}

Functoriality of the above pullback-construction in \ref{cleavageaction} gives us

\begin{proposition}\label{cleavagefunctor}
The construction of \ref{cleavageaction} defines an action of $\Cleav_{S^n}$ on $M^{S^n}$ as an element of the symmetric monoidal category $\left(\Corr(\Top),\times\right)$.
\end{proposition}

\begin{proposition}\label{resfibration}\label{kminusone}
The map $\res$ is a fibration.
\end{proposition}
\begin{proof}
Follows directly since the inclusions $\complement {S^n}_i \to S^n$ are closed cofibarations, and $\res$ is the dualization under the mapping space functor $M^-$.
\end{proof}


The rest of this section is devoted to give a hint at why the action through correspondences of $\Cleav_{S^n}$ gives rise to a homological action of $\h_*(\Cleav_{S^n})$ on $\hh_*(M^{S^n})$. As mentioned in the introduction, we give the full action in \cite[Th 1.1]{BargheerPuncture}; in this present section we only give examples and observations that can be seen as a justification of our definition of $\Cleav_{S^n}$ and a warm-up to the homological action.

 In obtaining homological actions of $\Cleav_{S^n}$ on $M^{S^n}$, there is a shift of degrees by the dimension of $M$, the following proposition gives the insight along spaces that this holds.

\begin{proposition}\label{constantnormal}
For any $[T,\underline{P}] \in \Cleav_{S^n}(-;k)$, with associated timber ${S^n}_1,\ldots,{S^n}_k$ the number
$$
\left|\pi_0\left(\coprod_{i=1}^k \complement {S^n}_i\right)\right| - \left|\pi_0\left(\beta_{[T,\underline{P}]}\right)\right|
$$
will be constantly $k-1$ for all $[T,\underline{P}] \in \Cleav_{S^n}(-;k)$.

\end{proposition}
\begin{proof}
Note that $\beta_{[T,\underline{P}]}$ constitute the boundary of the disjoint union of wedges of disks $\coprod_{i=1}^k \complement {S^n}_i$, so 
the number $\left|\pi_0\left(\coprod_{i=1}^k \complement {S^n}_i\right)\right|$ will be constant as long as $\left|\pi_0\left(\beta_{[T,\underline{P}]}\right)\right|$ is constant.

Consider a path $\gamma \colon [0,1] \to \Cleav_{S^n}(-;k)$ with $\left|\pi_0(\beta_{\gamma(0)})\right| = \left|\pi_0(\beta_{\gamma(1)})\right| + 1 $ and such that for a specific $t_0 \in [0,1]$, we have for all $t \leq t_0, \left|\pi_0(\beta_{\gamma(t)})\right| = \left|\pi_0(\beta_{\gamma(0)})\right|$ and for all $t > t_0, \left|\pi_0(\beta_{\gamma(t)})\right| = \left|\pi_0(\beta_{\gamma(1)})\right|$.

By definition of the cleaving proces, $\gamma(t_0)$ will have two hyperplanes $P_l,P_r$ such that $P_l \cap P_r \cap D^{n+1}$ is a nontrivial subspace of $S^n$, and for any  $\varepsilon > 0$, the same intersection for the hyperplanes of $\gamma(t_0 + \varepsilon)$ will be trivial; meaning that for sufficiently small $\varepsilon$, such that the hyperplanes do not become parallel, $P_l \cap P_r \subseteq \RR^{n+1}$ will be contained in $\RR^{n+1} \setminus D^{n+1}$. This has the effect that there is precisely one $j \in \{1,\ldots,k\}$ such that the complement of the timber indexed by $j$, $\complement D^{n+1}_j$, for $\gamma(t_0)$ has a connected component containing $P_l \cap P_r \cap D^{n+1}$, whereas for $\gamma(t_0 + \varepsilon)$ this becomes disconnected with different boundary components of $\complement {S^n}_i$ being formed using intersections with $P_l$ and $P_r$ respectively. Hence for these basic types of paths, an increase in $\left|\pi_0(\beta_{[T,\underline{P}]})\right|$ leads to an equal increase in $\left|\pi_0(\coprod_{i=1}^k \complement {S^n}_i)\right|$.

One can use these paths to parametrize a single cleaving hyperplane moving within $\Cleav_{S^n}(-;k)$, while the other $k-2$ hyperplanes remain fixed. From such parametrizations, one puts together a general path moving all hyperplanes of $\Cleav_{S^n}(-;k)$ and the result follows.

To compute the constant, take a configuration of hyperplanes where all hyperplanes are parallel, so that $\beta_{[T,\underline{P}]}$ has $k-1$ components. In this case the space of complements of the associated timber, $\coprod_{i=1}^k \complement {S^n}_i$ will have $(k-2)$ spaces $\complement {S^n}_i$ that consist of two disjoint spaces, and the two extremal complements that consist of a single subdisks of $S^n$. Hence in total $2(k-2) + 2 = 2(k-1)$ components. In effect, the constant will be $2(k-1) - (k-1) = k-1$. 

\end{proof}





\begin{remark}
Note that as we don't use it in the proof, \ref{constantnormal} holds even if we drop the assumption in \ref{NCleave} that for all elements of $\Cleav_{S^n}(-;k)$ the associated timber $S^n_1,\ldots,S^n_k$ should satisfy $\complement {S^n}_i \simeq \coprod_{\textrm{finite}} *$. In fact, this assumption is not needed for this paper, but will only be applied in \cite{BargheerPuncture} to define homological actions from the correspondences of \ref{cleavageaction}. 

\end{remark}

\begin{remark}\label{singleumkehr}
What we in fact show in \cite{BargheerPuncture} is something stronger than a homological action, namely that there is a stable action map, residing in the category of spectra:

\begin{eqnarray}\label{mentionstable}
\Cleav_{S^n}(-;k) \times \left(M^{S^n}\right)^k \to M^{S^n} \land S^{\dim(M)\cdot(k-1).}
\end{eqnarray}
Smashing the above map with the Eilenberg-Maclane spectrum, and taking homotopy groups yields the action mentioned just above Theorem B of the introduction.

We shall in \ref{twoaryex} below give an example of the action to illustrate the reasoning behind the definition of $\Cleav_{S^n}$. 

We shall first make some remarks on how the constructions of this section provides the foundation for how this map takes form.

First of all, note that since we in \ref{NCleave} have assumed that to $[T,\underline{P}] \in \Cleav_N(-;k)$ the associated timber $N_1,\ldots,N_k$ will have $\complement N_i$ consist of a finite disjoint union of contractible spaces. Therefore, in the diagram (\ref{basaction}), the space $\prod_{i=1}^k M^{\complement N_i}$ is a Poincar\'e duality space. That is, up to homotopy it is equivalent to a product of copies of $M$.

In the case $N = S^n$, taking such a specific $[T,\underline{P}]$ as a pointwise version of (\ref{mentionstable}), we should be able to obtain a map $\left(M^{S^n}\right)^k \to M^{S^n} \land S^{\dim(M) \times (k-1)}$. Indeed, there are methods available in the litterature to do this. For instance adapting the methods of \cite[p.14/'Umkehr maps in String Topology']{umkehrCohenKlein} provides such a map. Here it is crucial to note that the map $\phi$ is up to homotopy an embedding of codimension $\dim(M)\times(k-1)$ where the $(k-1)$-factor comes from \ref{kminusone}, which provides the dimension-shift in (\ref{mentionstable}).

Yet another method, formulated on the chain-level instead of spectra would be \cite[Theorem A]{StringGorenstein}. 

\end{remark}

The complexity of the construction of the umkehr map rises with the arity of the involved maps. As the arity rises, the map $\phi$ of (\ref{basaction}) will by \ref{constantnormal} have constant codimension, whereas the actual dimension will be exposed to sudden jumps in the actual dimensions of the spaces $\phi$ is mapping between, as is indicated in the proof of \ref{constantnormal}. The coherence issues of higher arity lies in patching the instances of such jumps together. The following example illustrates the case of arity $2$ operations where there are no such jumps:

\begin{example}\label{twoaryex}
The $2$-ary portion $\Cleav_{S^n}(-;2)$ is a manifold, specified by a single cleaving hyperplane, it deformation retracts onto $S^n$, which is determined by the direction of the normal-vector of the cleaving hyperplane.

Consider the pullback diagram
\begin{eqnarray}\label{twoary}
\xymatrix{
M_{\Cleav_{S^n}(-;2)}^{S^n} \ar[r]\ar[d] & \left(M^{S^n}\right)^2 \times \Cleav_{S^n}(-;2)\ar[d]|{\res}\\
M  \times \Cleav_{S^n}(-;2)\ar[r] & M^{\coprod_{i=1}^2 \complement {S^n}_i}_{\Cleav_{S^n}(-;2),}
}
\end{eqnarray}
where $\complement S^n_1$ and $\complement S^n_2$ denotes the complement inside $S^n$ of the timber associated to $[T,P] \in \Cleav_{S^n}(-;2)$. Considering it as a set, $M^{\coprod_{i=1}^2 \complement S^n_i }_{\Cleav_{S^n}(-;2)}$ is given by the  disjoint union $\coprod_{[T,P] \in \Cleav_{S^n}(-;2)} M^{\complement N_i}$. The map $\res$ in the diagram is given by letting $\res(f_1,f_2,[T,P])$ be the restriction of $f_1$ to $\complement S^n_1$ and $f_2$ to $\complement S^n_2$ along the component indexed by $[T,P]$. We topologize $M^{\coprod_{i=1}^2 \complement S^n_i}_{\Cleav_{S^n}(-;2)}$ by making $\res$ a quotient map.

Note that by \ref{resfibration}, pointwise in $\Cleav_{S^n}(-;2)$, $\res$ is a fibration. One sees that the lifts of this global map can be constructed to be continous in $\Cleav_{S^n}(-;2)$ as well. 

We hereby have homotopy-equivalences $$M^{\coprod_{i=1}^2 \complement S^n_i}_{\Cleav_{S^n}(-;2)} \simeq M^2 \times \Cleav_{S^n}(-;2) \simeq M^2 \times S^n$$ and $$M \times \Cleav_{S^n}(-;2) \simeq M \times S^n. $$ Stating that the lower portion of the diagram is an embedding of Poincar\'e duality spaces. For instance applying one of the methods mentioned in \ref{singleumkehr}, this hereby provides the first sign of an action of $\Cleav_{S^n}$ on $M^{S^n}$, and taking homotopy groups of this map, we get a map in homology

$$
\h_*\left(\Cleav_{S^n}(-;2)\right) \otimes \h_*(M^{S^n})^{\otimes 2} \to \h_{*+\dim(M)}(M^{S^n})
$$
\end{example}

\section{Spherical Cleavages are $E_{n+1}$-operads}\label{encomputation}
We now devote energy to prove that $\Cleav_{S^n}$ is a coloured $E_{n+1}$-operad. A concept defined in \ref{enoperaddefine}.





\subsection{Combinatorics of Coloured $E_n$-operads}
The following combinatorial data of full level graphs is the main tool we use to describe a detection principle for $E_n$-operads. 
\begin{definition}\label{fullgraphoperad}
By a \emph{full level $(n,k)$-graph}, we shall understand a graph $G$ with $k$ vertices $v_1,\ldots,v_k$ such that all pairs $(v_i,v_j)$ are connected by precisely one edge $e_{ij}$. Let $\mathbf{n} := \{0,\ldots,n-1\}$. We let each of the $k \choose 2$ edges of $G$ be labelled by elements of $\mathbf{n}$.

\end{definition}

We say that a full level $(n,k)$-graph $G$ is oriented if there to each edge $G$ $e_{ij}$ is designated a direction; either from $v_i$ to $v_j$ or from $v_j$ to $v_i$.

To $\sigma \in \Sigma_k$, there is a unique orientation of a full level $(n,k)$-graph, $G$. Namely by letting $e_{ij}$ point from $v_i$ to $v_j$ if $\sigma(i) < \sigma(j)$, and point from $v_j$ to $v_i$ if $\sigma(j) < \sigma(i)$. We call $\sigma$ the permutation associated to $G$. 

Indeed, assuming that this orientation of $G$ is oriented with no cycles and comes equipped with a sink and a source, one can reconstruct the permutation $\sigma_G$  associated to $G$ from the orientation of $G$: Let the index of the sink of $G$ be mapped to $k$ under $\sigma_G$, and successively remove the sink of an oriented full  $(n,i)$-graph with sink, source and no cycles to obtain a full $(n,i-1)$-graph with induced orientation being guaranteed a sink by the pigeonhole principle. The index of this sink is mapped to $i$ under $\sigma_G$. Continuing this process until only the source of $G$ is left makes $\sigma_G$ a well-defined permutation since $G$ has no cycles.

\begin{definition}\label{fullpermutegraph}
Let $\mathpzc{K}^n(k)$ denote the set of full level $(n,k)$-graphs, oriented via $\Sigma_k$ as above. This gives a bijection $\mathbf{n}^{k \choose 2} \times \Sigma_k \leftrightarrow \mathpzc{K}^n(k)$.
\end{definition}

We have that $\mathpzc{K}^n$ is an operad: To $G_k \in \mathpzc{K}^n(k)$ and $G_m \in \mathpzc{K}^n(m)$, inserting the $m$ vertices of $G_m$ instead of $i$'th vertex of $G_k$, we obtain a full level $(n,m+k-1)$ graph $G_k \circ_i G_m$, by labelling and orienting the edges of $G_k \circ_i G_m$ in the following way: 

Since we have replaced the $i$'th vertex of $G_k$ with $k$ vertices from $\mathpzc{K}^n(m)$, $G_k$ lives as a subgraph of $G_k \circ_i G_m$ in $k$ different ways -- one for each choice of replacement vertex for $i$ in $G_m$. The graph $G_m$ has all its vertices retained in $G_k \circ_i G_m$, so there is only one choice of subgraph for $G_m$.

We label and orient all edges of $G_k \circ_i G_m$ via the labellings and orientations of the possibilities of subgraphs $G_k$ and $G_m$.

An example operadic composition is given in figure \ref{fullgraphpic}

\begin{center}
\begin{figure}[h!tb]
\centerline{\includegraphics[scale=0.8]{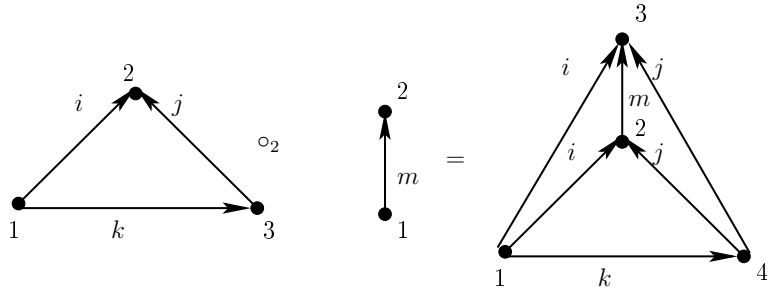}}
\caption{An operadic composition of a $3$-ary with a $2$-ary operation in the full-graph operad. The labellings $i,j,k,m$ are elements of $\mathbf{n}$. Note that in the $4$-ary operation both graphs are contained in the final result, and we copy the labelling of the edges that are going to the vertex the operadic composition is happening at}
\label{fullgraphpic}
\end{figure}
\end{center}



\begin{observation}\label{posetstructure}
Using \ref{fullpermutegraph}, we see how $\mathpzc{K}^n$ is an operad of posets. First of all, letting $\mathpzc{K}^n(2) = \mathbf{n} \times \Sigma_2$ be given by setting $\Sigma_2 = \{\id,\tau\}$ and partially ordering through the arrows of the diagram

\begin{eqnarray}\label{fullgraphorder}
\xymatrix{
(0,\id)\ar[r]\ar[dr] & (1,\id)\ar[r]\ar[dr] & \cdots\ar[r]\ar[dr] & (n,\id)\\
(0,\tau)\ar[r]\ar[ur] & (1,\tau)\ar[r]\ar[ur] & \cdots\ar[r]\ar[ur] & (n,\tau)
}
\end{eqnarray}

Consider the maps $\gamma_{ij} \colon \mathpzc{K}^n(k) \to \mathpzc{K}^n(2)$, sending $G_k \in \mathpzc{K}^n(k)$ to its subgraph, with one edge, spanned by the vertices $v_i$ and $v_j$. Following \cite[1.5]{BergerCellular}, we let the partial ordering be given by the coarsest ordering such that $\gamma_{ij}$ is order preserving for all $i< j \in \{0,\ldots,k\}$. 
\end{observation}

%

\begin{definition}
Given $\mathpzc{O}$ a coloured operad, let $S(\mathpzc{O})$ denote the coloured operad in posets, with objects the subsets of $\Ob(\mathpzc{O})$, and $k$-ary morphisms subsets of $\mathpzc{O}(-;k)$. The operad is an operad in posets through inclusions of subspaces.

That is, we let $\Ob(S(\mathpzc{O})) = \{0,1\}^{\Ob(\mathpzc{O})}$.
Let $S(\mathpzc{O})(-;k):= \{0,1\}^{\mathpzc{O}(-;k)}$. These fit into the diagram
$$
\xymatrix{
S(\mathpzc{O})(-;k+m-1) & S(\mathpzc{O})(-;m) \times_{\Ob(S(\mathpzc{O}))} S(\mathpzc{O})(-;k)\ar[l]_(.6){\circ_i}\ar[r]\ar[d] & S(\mathpzc{O})(-;k)\ar[d]|{\ev_i}\\
& S(\mathpzc{O})(-;m)\ar[r]^{\ev_{\inc}} & \Ob(S(\mathpzc{O}))
}
$$

Where as usual $\ev_i$ evaluates at the $i$'th colour of $S(\mathpzc{O})(-;k)$, given as a subset of $\Ob(\mathpzc{\mathpzc{O})}$ and $\ev_{\inc}$ evaluates the incoming colour of $S(\mathpzc{O})(-;m)$.

Operadic composition is induced from composition in $\mathpzc{O}$, pointwise. In the sense that $\circ_i \colon S(\mathpzc{O})(-;m) \times_{\Ob(S(\mathpzc{O}))} S(\mathpzc{O})(-;k) \to S(\mathpzc{O})(-;k+m-1)$ is given by the subset of $\mathpzc{O}(-;k+m-1)$ obtained by to any point of $\mathpzc{O}(-;m) \times_{\Ob(\mathpzc{O})} \mathpzc{O}(-;k)$ as an element of $S(\mathpzc{O})(-;m) \times_{\Ob(S(\mathpzc{O}))} S(\mathpzc{O})(-;k)$ applying the $\circ_i$-operation from $\mathpzc{O}(-;m) \times_{\Ob(\mathpzc{O})} \mathpzc{O}(-;k)$ to $\mathpzc{O}(-;k+m-1)$, and taking the union over the specific subset of these compositions.
\end{definition}

Recall from \ref{colouren} that to a monochrome operad $\mathpzc{P}$ and a space $X$, the coloured topological operad $\mathpzc{P} \times X$ is coloured over $X$ and the $k$-ary morphisms are formed by taking the cartesian product of $\mathpzc{P}(k)$ with $X$. 

The Berger Cellularization Theorem, written in a monochrome fashion in \cite[Th. 1.16]{BergerCellular} hereby transfers to our coloured setting:


\begin{theorem}\label{cellulite}
Let $\mathpzc{O}$ be a topological coloured operad. 
The operad $\mathpzc{O}$ is a coloured $E_n$-operad, \ref{colouren} if there is a functor $F_k \colon \mathpzc{K}^n(k) \times S(\Ob(\mathpzc{O})) \to S(\mathpzc{O})(-;k)$ that is, both a functor with respect to the poset structure as well as a morphism of coloured operads, satisfying the following:

\begin{itemize}

\item[(A)] Let $C_0 \in S(\mathpzc{\Ob(\mathpzc{O})}$. 
The \emph{latching space of $\alpha \in \mathpzc{K}^n(k)$} is given by $L_{(\alpha,C_0)} := \bigcup_{\beta < \alpha} F_k(\beta,C_0)$. We require that the morphism $$L_{(\alpha,C_0)} \hookrightarrow F_k(\alpha,C_0)$$ is a cofibration.
\item[(B)] For all $\alpha \in \mathpzc{K}^n(k)$, we require maps $F_k(\alpha,B) \to B$ where $B \in S(\Ob(\mathpzc{O}))$ such that these are weak equivalences, and natural with respect to morphisms $B \hookrightarrow C \in S(\Ob(\mathpzc{O}))$. 
\item[(C)] $\colim_{(\alpha,C_0) \in \mathpzc{K^n(k)} \times \Ob(S(\mathpzc{O}))}F_k(\alpha,C_0) = \mathpzc{O}(-;k)$, where the colimit is using the poset-structure on $\mathpzc{K}^n(k)$ -- given in \ref{posetstructure} -- and inclusions of subsets in $S(\Ob(\mathpzc{O}))$. These inclusions should be compatible with the equation, in the sense that $\mathpzc{O}(U;k)$ should be given by restricting the colimit to the $C_0 \in \Ob(S(\mathpzc{O}))$ satisfying $C_0 \subseteq U$ in the indexing category.
\end{itemize}
\end{theorem}
\begin{proof}

First of all, note that $\mathpzc{K}^n(k)$ as a finite poset is a Reedy Category in the sense of \cite[13.1]{DuggerPrimer}; explicitly a degree function $\deg \colon \mathpzc{K}^n(k) \to \ZZ$ can be given by letting $\deg(\alpha)$ be determined by the sum of the $k \choose 2$ labels in $\{0,\ldots,n-1\}$ of the edges of the graph $\alpha$. 



From the assumptions (A)-(C) we get the following homotopy equivalence: 

\begin{eqnarray*}
\mathpzc{O}(-;k) = \colim_{(\alpha,C_0) \in \mathpzc{K}^n(k) \times S(\Ob(\mathpzc{O}))} F_k(\alpha,C_0) &\cong&\\ \colim_{C_0 \in S(\Ob(\mathpzc{O}))}\colim_{\alpha \in \mathpzc{K}^n(k)}F_k(\alpha,C_0) &\simeq&\\
 \colim_{C_0 \in S(\Ob(\mathpzc{O}))}(\hocolim_{\alpha \in \mathpzc{K}^n(k)}* \times C_0) &\cong&\\ \colim_{C_0 \in S(\Ob(\mathpzc{O}))}(|\mathcal{N}(\mathpzc{K}^n)| \times C_0)(k) &\cong&\\ (|\mathcal{N}(\mathpzc{K}^n)| \times \Ob(\mathpzc{O}))(k) \end{eqnarray*}
The first identification is given in assumption (C), which in turn splits out into a double colimit. To obtain the homotopy equivalence: Since $\mathpzc{K}^n(k)$ is a Reedy Category, the assumption (A) allows us to apply \cite[13.4]{DuggerPrimer}, or in published format: \cite[15.2.1]{HirschhornModelCat}, giving a homotopy equivalence $\hocolim_{\alpha \in \mathpzc{K}^n(k)} F_k(\alpha,C_0) \to \colim_{\alpha \in \mathpzc{K}^n(k)} F_k(\alpha,C_0)$ for fixed incoming colours $C_0 \in S(\Ob(\mathpzc{O}))$. From (B) we get a homotopical identification of $F_k(\alpha,C_0)$ with $C_0$ and this computes the homotopy colimit, geometrically realizing the nerve of the full graph operad, along with a cartesian product of the colours $C_0 \in S(\Ob(\mathpzc{O}))$ -- independent of $\alpha \in \mathpzc{K}^n(k)$. The final identification follows since the naturality of (B) supplies us with a cartesian product of the nerve along with an actual direct limit of all inclusions of $S(\Ob(\mathpzc{O}))$ which can be identified with the final target of the inclusions, $\Ob(\mathpzc{O})$.



One now utilizes $F$ as a morphism between coloured operads to check that this gives an operadic weak equivalence $\mathpzc{O} \simeq |\mathpzc{N}(\mathpzc{K}^n)| \times \Ob(\mathpzc{O})$.


\end{proof}

\begin{definition}
Let $\mathpzc{O}$ be a coloured topological operad. We call $F \colon \mathpzc{K}^n \times S(\Ob(\mathpzc{O})) \to S(\mathpzc{O})$ an \emph{$E_{n}$-functor} if it satisfies the conditions of \ref{cellulite}.
\end{definition}

\begin{remark}\label{knowen}
In order to get an equivalence back to something known, let $\Disk_{n+1}$ denote the little disk operad. We hereby have that the coloured operad $\Disk_{n+1} \times \Ob(\Cleav_{S^n})$ is a coloured $E_{n+1}$-operad in the sense of \ref{enoperaddefine}, coloured over the same objects as $\Cleav_{S^n}$.

\end{remark}

\subsection{An $E_{n+1}$-functor for $\Cleav_{S^n}$}



\begin{construction}\label{basetrees}
We shall provide the combinatorial data, giving the link between the full graph operad, and the spherical cleavage operads.

For each $n \in \NN$, let $I = [-1,1] \hookrightarrow \RR^{n+1}$ denote the interval as sitting inside the first coordinate axis of $\RR^{n+1}$. Choosing $k-1$ distinct points $x_1,\ldots,x_{k-1} \in I$ specifies a partition of $I$ into $k$ intervals $X_1=[-1,x_1],X_2=[x_1,x_2],\ldots,X_k=[x_{k-1},1]$. We endow the collection of these intervals with an ordering, determined by $\sigma \in \Sigma_k$, ordering them as $X_{\sigma(1)},\ldots,X_{\sigma(k)}$. 

Parametrize $S^n := \{\underline{s} = (s_1,\ldots,s_{n+1}) \in \RR^{n+1} \mid \|\underline{s}\| = 1\}$ to consider the map $\eta \colon S^n \to I$ given by $\eta(s_1,\ldots,s_{n+1}) = s_1$. Any subinterval $X_i \subseteq I$ defines timber $\tilde{X_i}$ of $S^n$, by $\eta^{-1}(X_i)$.

Positioning hyperplanes $P_1,\ldots,P_{k-1}$ orthogonal to $I$ -- such that $P_i$ contains the point $(x_i,0,\ldots,0)$, as decorations on a cleaving tree we can choose their normal-vector of $P_i$ to point towards $(1,0,\ldots,0)$ and get colours that are labelled by $1$ to $k$ from left to right along the first coordinate axis. 

Under the chop-equivalence of \ref{NCleave}, we can always choose a representing cleaving tree with two leaves labelled $i$ and $i+1$ directly above an internal vertex for any $i \in \{1,\ldots,k-1\}$. For this particular representative of a cleaving tree with the particular orientations of $P_1,\ldots,P_{k-1}$, applying the transposition between $i$ and $i+1$ corresponds exactly to inverting the orientation of the decoration at the vertex below the two leafs.

Since $\Sigma_k$ is generated by these transpositions, we can permute the labelling of the colours by $\sigma \in \Sigma_k$, and hereby obtain $[T_\sigma,\underline{P_\sigma}] \in \Cleav_{S^n}(Uk)$, with $U \in \Ob(\Cleav_{S^n})$ and outgoing colours decorated by $\sigma(1),\ldots,\sigma(k)$ from left to right along the first coordinate axis of $\RR^{n+1}$. 





\end{construction}

\begin{definition}\label{famhyp}
For a given $\sigma \in \Sigma_k$, the collection of all $[T_\sigma,\underline{P}_\sigma]$ as given above specifies an element of the subsets of $\Cleav_{S^{n}}(-;k)$, where for $U \in \Ob(\Cleav_{S^n})$ this involves a restriction to the $U$-cleaving trees $(T_{\sigma},\underline{P}_{\sigma})$. 

Note that to $U \in \Ob(\Cleav_{S^n})$ choosing the hyperplanes $\underline{P}_{\sigma}$ to be in equidistant position from each other, and the closest hyperplanes that no longer cleave $U$, defines an embedding of $U \hookrightarrow \Cleav_{S^n}(U;k)$ for each $\sigma \in \Sigma_k$.
\end{definition}

\begin{observation}\label{intervaliscleaving}
We define $J_U \subseteq [-1,1]$ as sitting inside the first coordinate axis of $\RR^{n+1}$ as the subspace where any $U$-cleaving $(T_{\sigma},\underline{P}_{\sigma})$ have decorating hyperplanes contain points of $J_U$.  

We shall for the sake of this section allow ourselves to assume that $J_U \subseteq [-1,1]$ is a non-empty subinterval; formally, this can be done by redefining $\Cleav_{S^n}$ as a full suboperad of $\Cleav_{S^n}$ given by restricting $\Ob(\Cleav_{S^n})$ to the timber $U$ for which $J_U$ is $U$-cleaving.

Similar to \ref{contractibletimber}, one can define a homotopy that pushes the hyperplanes defining $U$ towards tangent-hyperplanes of $S^n$ to show that this restriction defines a deformation retraction of the objects, and hence makes the inclusion a weak equivalence of operads.

\end{observation}

\begin{remark}
We find it enlightening to note that we can form a suboperad, the \emph{caterpillar operad} \footnote{collapse the components of the blueprint to a single point, to make a visual link to this name} $\Cater_{S^n}$, of $\Cleav_{S^n}$ by taking the full suboperad under the condition that $[T,\underline{P}] \in \Cater_{S^n}(-;k)$ if $[T,\underline{P}]$ is of the form $[T_\sigma,\underline{P}_{\sigma}]$ as in \ref{famhyp} for some $\sigma \in \Sigma_k$.

The caterpillar operad will control the product structure on $\hh_*(M^{S^n})$, however in order to obtain the higher bracket in the Gerstenhaber Algebra, we shall need more than just parallel hyperplanes -- and engage all the ways hyperplanes can rotate in $\Cleav_{S^n}$, in contrast to the sole translational data of $\Cater_{S^n}$.

Said in a different way, there is an obvious operadic map from $\Cater_{S^n}$ to the little intervals operad, see e.g. \cite[ch. 2]{McClureSmithcosimplicial}, determined by how the hyperplanes of $\Cater_{S^n}$ partition the $x$-axis into intervals. The little intervals operad has the $k$-ary space given as the space of embeddings of $k$ intervals inside $[0,1]$. We can expand these little intervals linearly until they touch each other, and hereby similarly to $\Cater_{S^n}(-;k)$ partitioning $[0,1]$ into $k$ smaller intervals. This provides a map that is a weak equivalence of coloured operads from $\Cater_{S^n}$ to the little intervals operad, considered as a coloured operad with trivial colours. This hence provide the $A_\infty$- or $E_1$-structure on the String Product associated to $M^{S^n}$ for the action in spectra that is constructed in \cite{BargheerPuncture}. The rest of this section is hence devoted to determining the rest of the $E_{n+1}$-structure of $\Cleav_{S^n}$.
\end{remark}




In order to engage the combinatorics of this rotational data, i.e. define a $E_{n+1}$-functor, we shall prescribe explicit transformations of the hemispheres parametrizing the hyperplanes involved in the cleavages.





\begin{definition}\label{simplePhi}
We prescribe a function $\kappa \colon S^n \times \RR \to \Hyp^{n+1}$, in accordance with \ref{hyperplanedef} 
by letting $\kappa(s,t)$ be given as the oriented hyperplane that has $s$ as a normal vector, and has been translated by $t$ -- to contain the point $s\cdot t\in \RR{n+1}$.

\end{definition}

\begin{definition}\label{hemispheres}
In the following, when referring to a sphere $S^i$, we shall generally consider it as sitting inside a string of inclusions 
\begin{eqnarray}\label{equatorinclusions}
\xymatrix{
S^0 \ar@{^(->}[r]^{\iota^0} &  S^1 \ar@{^(->}[r]^{\iota^1}& \cdots  \ar@{^(->}[r]^{\iota^{n-1}} &S^{n-1} \ar@{^(->}[r]^{\iota^{n-1}} & S^{n}
}
\end{eqnarray}
where all are subsets of $\RR^{n+1}$, and $S^{i-1}$ is embedded equatorially into the first $i$ coordinates of $S^i \subset \RR^{i+1}$.

Let $S^i_+ = \{(x_0,\ldots,x_{n}) \in S^n \mid x_i \geq 0, x_{i+1} = \cdots = x_{n-1} = 0\}$ and $S^i_- = \{(x_0,\ldots,x_{n}) \in \RR^{n+1} \mid x_i \leq 0,x_{i+1} = \cdots = x_{n-1} = 0\}$. 
Restrictions of $\iota_i$ in (\ref{equatorinclusions}), yields the partially ordered set of inclusions:

\begin{eqnarray}\label{hemisphereinc}
\xymatrix{
S^0_+ \ar@{^(->}[r]\ar@{^(->}[dr] &  S^1_+ \ar@{^(->}[r]\ar@{^(->}[dr] & \cdots  \ar@{^(->}[r] \ar@{^(->}[dr]&S^{n-1}_+ \ar@{^(->}[r] \ar@{^(->}[dr]& S^{n}_+\\
S^0_- \ar@{^(->}[r]\ar@{^(->}[ur] &  S^1_- \ar@{^(->}[r] \ar@{^(->}[ur]& \cdots  \ar@{^(->}[r] \ar@{^(->}[ur]&S^{n-1}_- \ar@{^(->}[r] \ar@{^(->}[ur]& S^{n}_-}.
\end{eqnarray}
Since these are all inclusions of closed lower-dimensional submanifolds, it is a partially ordered set of cofibrations. 


\end{definition}

\begin{definition}\label{stringindex}
Consider the partially ordered set $I_{k}$ with objects $\alpha_j \subseteq \{1,\ldots,k\}$, where $j$ indicates that $\alpha_j$ is of cardinality $j$, and morphisms generated by the opposite arrows of \emph{simple inclusions} $\alpha_{j-1} \to \alpha_j$.

To $f \in I_{k}$, where $f \colon \alpha_j \to \alpha_p$ let the domain be denoted by $D(f):= \alpha_j$ denote the domain of $f$, and the target $T(f) := \alpha_p$. A simple inclusion $\iota_l$ defines a \emph{lost number} $j_{\iota_l} := T(\iota_l) \setminus D(\iota_l) \in \{1,\ldots,k\}$. 

An \emph{$i$-string} of morphisms in $I_{k}$ is given by a sequence $\underline{\iota} = (\iota_1,\ldots,\iota_{k-1})$ of opposite arrows of simple inclusions such that $D(\iota_r)  = T(\iota_{r+1})$ -- and with $D(\iota_1) = \{1,\ldots,k\}$ and $T(\iota_{k-1}) = \{i\}$. 

Let $I_{k}|_i$ denote the set of $i$-strings of $I_{k}$.

\end{definition}

\begin{remark}\label{stringpermute}
While we shall mainly find it convenient to use the notation of \ref{stringindex}, note that the data of $\underline{\iota} \in I_{k}|_i$ exactly corresponds to a permutation of $\{1,\ldots,k\}$, where we to $i$ assign the lost number of $\iota_i$.

We shall thus allow ourselves to consider $\underline{\iota}$ as an element of $\Sigma_k$ where we here use the notation $\underline{\iota}(j) \in \{1,\ldots,k\}$ to indicate the value of $j$ under the permutation.
\end{remark}

\begin{construction}\label{techfunc}
We conglomerate the above constructions into a specific recursively defined function that provide the technical core in the definition of the $E_n$-functor.

For each $i \in \{1,\ldots,k-1\}$, we want to define a function

$$
\Theta_U \colon \Sigma_k \times I_{k-1}|_i \times \left(R_U \cap \left(S^n \times \RR\right)^{k-1}\right) \to \Cleav_{S^n}(U;k).
$$

Where $R_U \subseteq \left(S^n \times \RR\right)^{k-1}$, amounts to a corestriction of each $(S^n \times \RR)$-factor that will be specified below.



Each $\iota_i$ will specify a hyperplane via its lost number, \ref{stringindex}, as $P_{j_{\iota_{i}}}$ of $\underline{P}_{\sigma}$. 

We shall produce a $U$-cleaving tree that is decorated by the hyperplanes $$\kappa(s_1,r_1), \ldots,\kappa(s_{k-1},r_{k-1}).$$

In order to specify the $U$-cleaving tree that these hyperplanes decorate, we utilize the ordering from left to right along the $x$-axis of the hyperplanes of $\underline{P}_{\sigma}$, specified in \ref{basetrees}, as well as the $i$-string $\underline{\iota}$. 

We build this tree recursively, and start by positioning $\kappa(s_1,r_1)$ as the decoration of a $2$-ary tree $T_2$. We hereby restrict $\Theta_U$ by letting the first $\left(S^n \times \RR\right)$-factor of $R_U$ be such that $(T_2,\kappa(s_1,r_1))$ is $U$-cleaving.

In the recursive step, assume that we have defined the first $l-1$ factors of $R_U$ and that we are given an $l$-ary $U$-cleaving tree $(T_l,\kappa(s_1,r_1),\ldots,\kappa(s_{l-1},r_{l-1}))$ such that taking $l-1$ hyperplanes of $\underline{P}_{\sigma}$, $P_{j_{\iota_1}},\ldots,P_{j_{\iota_{l-1}}}$, and assigning $P_{j_{\iota_i}}$ to replace the decoration $\kappa(s_i,r_i)$, also yields a $U$-cleaving tree.

The hyperplane $P_{j_{\iota_l}}$ cleaves timber associated to a specific leaf of the decorated tree $(T_l,P_{j_{\iota_1}},\ldots,P_{j_{\iota_{l-1}}})$. We graft a $2$-ary tree onto $T_l$ at this leaf, to obtain the $(l+1)$-ary $T_{l+1}$. We let $\kappa(s_{l},r_l)$ be the decoration at the new internal vertex of $T_{l+1}$, where we define the $l$'th factor of $R_U$ by requiring that $(s_l,r_l) \in S^n \times \RR$ makes the decorated $(T_{l+1},\kappa(s_1,r_1),\ldots,\kappa(s_l,r_l))$ $U$-cleaving. The timber at the leaves of the decorated tree $(T_{l+1},\kappa(s_1,r_1),\ldots,\kappa(s_l,r_l))$ are induced by the timber at the same leafs of $(T_{l+1},P_{j_{\iota_1}},\ldots,P_{j_{\iota_{l}}})$.

\end{construction}

Any edge $e_{ij}$ of a graph $G \in \mathpzc{K}^n(k)$ is uniquely determined as the edge attached to the vertices labelled by some ordered pair $i<j \in \{1,\ldots,k\}$. Let $\omega_G(i,j) \in \ZZ_2 =: \{'+','-'\}$ be given by $\omega_G(i,j) = '+'$ if $e_{ij}$ points from $i$ to $j$ and $\omega_G(i,j) = '-'$ if $e_{ij}$ points from $j$ to $i$.

Denote by $\lambda_G(i,j) \in \{0,\ldots,n-1\}$ the labelling of the edge $e_{ij}$. 

\begin{construction}\label{enfunctorconstruction}
We shall construct an $E_{n+1}$-functor for $\Cleav_{S^{n}}$. That is, we are after a functor $D \colon \mathpzc{K}^{n+1} \times S(\Ob(\Cleav_{S^{n}})) \to S(\Cleav_{S^{n}})$.

Let $G \in \mathpzc{K}^{n+1}(k)$ be a graph with underlying permutation given by $\sigma_G \in \Sigma_{k}$



Let a $i$-string $\underline{\iota} \in I_{k-1}|_i$ be given, and let $l \in \{1,\ldots,k-1\}$. We consider the two lost numbers $j_{\iota_l}$ and $j_{\iota_{l-1}}$ in the sense of \ref{stringindex}. To account for the case $l=1$, we let $j_{\iota_{0}} := k$. Denote by $\iota_{l_{\max}} := \max\{j_{\iota_l},j_{\iota_{l-1}}\}$ and $\iota_{l_{\min}} := \min\{j_{\iota_l},j_{\iota_{l-1}}\}$. 


We define the $k$'th operadic constituent of $D$ as:

$$
D_{k}(G,A_0) = 
$$
\begin{eqnarray}\label{DDefine}
\bigcup_{a_0 \in A_0}\bigcup_{i \in \{1,\ldots,k-1\}}\bigcup_{\underline{\iota} \in I_{k-1}|_i}\Theta_{a_0}\left(\sigma_G,\underline{\iota},\left(R_{a_0}(\iota_1),\ldots,R_{a_0}(\iota_{k-1})\right)\right),
\end{eqnarray}

where we to $\iota \in I_{k-1}|_i$ let the spaces $R_U(\iota_{l})$ be a further restriction of the spaces $S^{\lambda_G(\iota_{l_{\min}},\iota_{l_{\max}})}_{\omega_G(\iota_{l_{\min}},\iota_{l_{\max}})} \times \RR$ considered as the $i$'th input to $\Theta_U$, that is, a restriction of the space $R_U$ of \ref{techfunc}. 

The further restriction is given by restricting to the pathcomponent of $[T_{\sigma_G},P_{\sigma_G}]$

Here, the hemispheres $S^{\lambda_G(\iota_{l_{\min}},\iota_{l_{\max}})}_{\omega_G(\iota_{l_{\min}},\iota_{l_{\max}})}$ are given in the diagram (\ref{hemisphereinc}).








\end{construction}


\subsection{Proof that Cleavages are $E_{n+1}$}



To state the first lemma, note that maps $G \to G' \in \mathpzc{K}^n(k)$ given by raising the index of an edge of $G$ from $i$ to $l$ where $i < l$, will induce injective maps $D(G,A_0) \hookrightarrow D(G',\underline{A})$ given by a restriction of the inclusion $S^i \hookrightarrow S^l$ as one of the coordinates of $R_U(\iota_j)$ under $D_k(G,A_0) \to D_k(G',A_0)$. This describes how $D$ is a functor of posets, and we use the following three lemmas to check the conditions (A)-(C) of \ref{cellulite} to prove that $\Cleav_{S^n}$ is a coloured $E_{n+1}$-operad.

\begin{lemma}\label{cofiberlemma}
As in (A) of \ref{cellulite}, consider the latching space $L_{(G,A_0)} = \bigcup_{G' < G} D_k(G',A_0)$. The induced map $L_{(G,A_0)} \hookrightarrow D_k(G',A_0)$ is a cofibration 
\end{lemma}
\begin{proof}
The inclusions of submanifolds in (\ref{hemisphereinc}) are of codimension strictly larger than 0, so these are automatically cofibrations. The maps out of $D_k(G',A_0)$ are built out of these maps by restricting to $a_0 \in A_0$-cleaving trees, along with pushouts and factors of cartesian products. Since whether a decoration of a $a_0$-cleaving tree is cleaving or not is an open condition (that is, if it holds for the hyperplane it holds for a small neighborhood of the hyperplane), the associated restriction of inclusions induced from (\ref{hemisphereinc}) will again be an inclusion of submanifolds that are of codimension greater than 0. Each $D_k(G',A_0) \hookrightarrow D_k(G,A_0)$ will hence result in a cofibration, that can be obtained as a lower-dimensional skeleton of a CW-structure on $D_k(G,A_0)$. Since the latching space is given by a finite union of these lower-dimensional spaces, the map from the latching space is again a cofibration.
\end{proof}

\begin{lemma}\label{limitlemma}
$\colim_{(G,A_0) \in (\mathpzc{K}^{n+1} \times \Ob(\Cleav_{S^n}))(k)} D_k(G,A_0) = (\Cleav_{S^{n}})(-;k)$
\end{lemma}
\begin{proof}
Any $[T,\underline{P}] \in \Cleav_{S^n}(-;k)$ can be obtained as an element of (\ref{DDefine}) for some choice of hemispheres determined by $G \in \mathpzc{K}^n(k)$. Note namely that the definition of $\Theta_U$ is a function that exactly mimics the cleaving procedure above \ref{cleaveconditions}, and therefore will$[T,\underline{P}]$ as obtained by this cleaving procedure be obtained since the hemispheres involved in the image of $D_k(G,A_0)$ cover $S^n$, and since we in \ref{enfunctorconstruction} are taking of path-components of $[T_{\sigma},P_{\sigma}]$ for all $\sigma \in \Sigma_k$.

\end{proof}


We say that for $(T,\underline{P})$ an $S^n$-cleaving tree that the hyperplane of the decoration $P_j$ \emph{dominates} another hyperplane $P_l$ of the decoration of $T$ if $P_j$ and $P_l$ intersect within $D^{n+1}$ and there are points of $P_l$ that lie on $\beta_{[T,\underline{P}]}$ and on one of the subspaces of $\RR^{n+1}$ that has been bisected by $P_j$, but none on the other side.

\begin{lemma}\label{contractlemma}
Given $A_0 \in S(\Ob(\Cleav_{S^n}))$ and $G \in \mathpzc{K}^{n+1}(k)$, there is a homotopy equivalence $D_k(G,A_0) \simeq A_0$, where $A_0$ is considered as a subspace of $\Cleav_{S^n}(-;k)$ by \ref{famhyp}.
\end{lemma}
Another usage of terminology would be that we supply a deformation retraction onto $A_0$, that is not a strong deformation retraction, in that we supply a homotopy $F \colon \Cleav_{S^n}(-;k) \times [0,1] \to \Cleav_{S^n}(-;k)$, where we only for $t \in \{0,1\}$ guarantee $F(a_0,t) = a_0$ for $a_0 \in A_0 \subset \Cleav_{S^n}(-;k)$ included as above. 
\begin{proof}
We break the proof, specifying the homotopy into three steps. Our overall strategy will be to in the first step show that the effect of different elements $a_0$ of $\Ob(\Cleav_{S^n})$ as input colours to $D_k$ can be neglected, as we similar to the proof of \ref{contractibletimber} can push the defining hyperplanes towards tangent-hyperplanes of $S^n$. 

The second step will show that for a given $\iota \in I_{k-1}|_i$, the space $$
\Theta_{a_0}\left(\sigma_G,\underline{\iota},\left(R_{a_0}(\iota_1),\ldots,R_{a_0}(\iota_{k-1})\right)\right)$$ given as one of the constituents in the union of (\ref{DDefine}) is contractible. Here, we use the $\iota$ fixing an ordering of the hyperplanes cleaving $S^n$. Using this ordering of hyperplanes, we can reparametrize them individually until they all align orthogonally to the first coordinate axis, which hence allow us to reduce back to step 1; dealing with the case of hyperplanes orthogonal to the first coordinate axis. 

In the final step, we show that the contractible spaces of step 2 are all glued together along a single contractible space containing the element $[T_{\sigma_G},\underline{P_{\sigma_G}}]$, hence making the entire $D_k(G,A)$ contractible.

\begin{bf}Step 1:\end{bf} Assume that the labelling of the edges of $G \in \mathpzc{K}^{n+1}(k)$ satisfy $\lambda_G(i,j) = 0$ for all $i < j \in \{1,\ldots,k\}$, i.e. all edges of $G$ are labelled by $0$. In this case, we have that $D_k(G,A_0)$ will pointwise in $A_0$ be the space of all sets of $k-1$ hyperplanes $P_1,\ldots,P_{k-1}$ orthogonal to the first coordinate axis of $\RR^{n+1}$ that cleaves $a_0 \in A_0$, where the parallel hyperplanes are ordered as $P_{\sigma_G}$ of \ref{basetrees}, and $\sigma_G \in \Sigma_k$ is the permutation associated to $G$.

The space $D_k(G,A_0)$ is homotoped onto $A_0$ by considering the interval $J_{a_0} \subseteq [-1,1]$ of the first coordinate axis of $\RR^{n+1}$ as given in \ref{intervaliscleaving}. Homotoping $P_1,\ldots,P_k$ to be equidistant within $J_{a_0}$ for each $a_0 \in A_0$ yields a deformation retraction onto $A_0$ in the sense of \ref{famhyp}, since the topology of $\Ob(\Cleav_{S^n})$ as determined by hyperplanes forming the timber lets the endpoints of $J_{a_0}$ vary continuously as functions of $\Ob(\Cleav_{S^n})$ to $\RR^2$, making the equidistance of hyperplanes continous as well.

\begin{bf}Step 2:\end{bf} For a general $G \in \mathpzc{K}^n(k)$ and a fixed $i$-string $\underline{\iota}$, and $a_0 \in A_0$, we see that the space $\Theta_{a_0}(\sigma_G,\underline{\iota},(R_{a_0}(\iota_1),\ldots,R_{a_0}(\iota_{k-1})))$ is weakly equivalent to a product of hemispheres, considered as a subspace of $\Cleav_{S^n}(-;k)$.

First of all, note that we can homotope $a_0$ in a similar fashion to \ref{contractibletimber}, bringing the hyperplanes defining $a_0$ so close to tangent-hyperplanes that they all intersect nontrivially within $S^n$. Call the result of this deformation $a_1 \in \Ob(\Cleav_{S^n})$. The timber $a_1$ provides extra freedom of movement of the hyperplanes given by $\Theta_{a_1}$.

We homotope $\Theta_{a_1}(\sigma_G,\underline{\iota},(R_{a_1}(\iota_1),\ldots,R_{a_1}(\iota_{k-1})))$ by using the ordering of dominance of hyperplanes as specified by $\underline{\iota}$. The least dominant hyperplane will be $\kappa(s_{k-1},r_{k-1})$, which in turn corresponds to a hyperplane $P_{k-1}$ of the parallel hyperplanes $\underline{P_{\sigma_G}}$. If $P_{k-1}$ is not the left-most hyperplane in the configuration of hyperplanes, then there is a hyperplane $P'$ to the immediate left of $P_{k-1}$. If $P_{k-1}$ is the leftmost hyperplane, then let $P'$ be the hyperplane immediately to the right of $P_{k-1}$.

The configuration specified by $\Theta_{a_1}$ will be such that the hyperplane $\kappa(s_j,r_j)$ corresponding to $P'$ will bound the same timber as $\kappa(s_{k-1},r_{k-1})$. This means that there is a path in the subspace of $\left(S^n \times \RR\right)^2$ parametrizing $\kappa(s_{k-1},r_{k-1})$ and $\kappa(s_l,r_l)$ such that $\kappa(s_{k-1},r_{k-1})$ can be brought to be parallel with $\kappa(s_l,r_l)$. This path can be chosen to be geodesic along the $S^n$-factor parametrizing $\kappa(s_{k-1},r_{k-1})$, and uses the freedom of the $\RR^{n+1}$-factor for the parametrizing subspace of both $\kappa(s_{k-1},r_{k-1})$ and $\kappa(s_l,r_l)$ to, if necessary, bring these hyperplanes closer to tangent-hyperplanes of $S^n$. The freedom of bringing these hyperplanes closer to tangent hyperplanes ensures that any reparametrization by $s' \in S^n$ of $\kappa(s_{k-1},r_{k-1})$ to $\kappa(s',r_{k-1}')$ can have the associated configuration of hyperplanes cleaving.

With $\kappa(s_{k-1},r_{k-1})$ and $\kappa(s_j,r_j)$ parallel, they can be translated as close to each other as suited and hereby for all practical purposes count as a single cleaving hyperplane. Therefore, the process now iterates by taking $\kappa(s_{k-2},r_{k-2})$ and similarly bringing it parallel to the nearby hyperplane as specified above. We continue this process until all the hyperplanes have been brought to be parallel, and they hence can be brought to be orthogonal to the first coordinate axis of $\RR^{n+1}$. 

\begin{bf}Step 3:\end{bf} 
Given $\underline{\iota}$ a $i$-string, and $\underline{\lambda}$ a $l$-string, we show that the associated spaces of $\underline{\iota}$ and $\underline{\lambda}$, $$\Theta_{a_0}(\sigma_G,\underline{\iota},(R_{a_0}(\iota_1),\ldots,R_{a_0}(\iota_{k-1}))) \textrm{ and } \Theta_{a_0}(\sigma_G,\underline{\lambda},(R_{a_0}(\lambda_1),\ldots,R_{a_0}(\lambda_{k-1})))$$ respectively, are glued together along contractible subsets, where we consider these as subsets of $\Cleav_{S^n}(-;k)$. Since step 2 tells us that the space associated to $\underline{\iota}$ and the one associated to $\underline{\lambda}$ are contractible, we have that the glued spaces are contractible. This follows combining the Mayer-Vietoris sequence, Hurewicz map and the Van Kampen theorem so that if $A$ and $B$ are two contractible spaces and $A \cap B$ is contractible then $A \cup B$ is again a contractible space. As we shall see, all spaces associated to elements of $\bigcup_{i=1}^{k-1} I_i$ will be glued together along the common basepoint determined by the hyperplanes $\underline{P}_{\sigma_G}$ orthogonal to the first coordinate axis of $\RR^{n+1}$ as determined by the orientation of $G$. In effect, the glueing of all these spaces will result in a contractible space.

In order to see how the spaces are glued together, we see in \ref{techfunc} that the spaces associated to $\underline{\iota}$ and $\underline{\lambda}$ are given by choosing new normal-vectors for the hyperplanes $P_1,\ldots,P_k$ of $\underline{P}_{\sigma_G}$ and decorating them on different trees. 

Considering $\underline{\iota}$ and $\underline{\lambda}$ as permutations in $\Sigma_k$ as given in \ref{stringpermute}, it a necessary but not sufficient condition that points in the spaces associated to $\underline{\iota}$ and $\underline{\lambda}$ determined by the tuples $((s_1,r_1),\ldots,(s_{k-1},r_{k-1})$ and $((s_1',r_1'),\ldots,(s_{k-1}',r_{k-1}'))$ satisfy
\begin{eqnarray}\label{necccleave}
(s_{\underline{\iota}^{-1}(i)},r_{\underline{\iota}^{-1}(i)}) = (s_{\underline{\lambda}^{-1}(i)}',r_{\underline{\lambda}^{-1}(i)}')
\end{eqnarray}
for all $i \in \{1,\ldots,k-1\}$, since the associated timber of the cleaving trees have to agree.

The condition (\ref{necccleave}) is not sufficient since the way the hyperplanes dominate each other might be different according to the two trees that the hyperplanes are decorating. Note first of all that if (\ref{necccleave}) is satisfied, and $s_{\underline{\iota}^{-1}(i)} = s_{\underline{\lambda}^{-1}(i)}$ is the point $S^0_{\omega(\iota(i)}$ all the hyperplanes are parallel and orthogonal to the first coordinate axis, where there is no dominance amongst hyperplanes, identifying the points.

We have an ordering on the hyperplanes of the spaces associated to $\underline{\iota}$ and $\underline{\lambda}$, given by $P_{\underline{\iota}(1)},\ldots,P_{\underline{\iota}(k-1)}$ and $P_{\underline{\lambda}(1)},\ldots,P_{\underline{\lambda}(k-1)}$. Assume we are given $[T,P_{\underline{\iota}}]$ and $[T',P_{\underline{\lambda}}]$ satisfying (\ref{necccleave}) such that they agree as elements of $D_k(G,A_0)$. We want to show that there is a unique path homotoping them onto trees decorated by hyperplanes orthogonal to the first coordinate axis, in which case step 1 applies to provide the deformation retraction.

For the hyperplane $P_{\underline{\iota}(1)}$, no other hyperplanes $P_{\underline{\iota}(2)},\ldots,P_{\underline{\iota}(k-1)}$ will dominate $P_{\underline{\iota}(1)}$. 

We can use step 2 to assume that for the 2-ary tree decorated by only $P_{\underline{\iota_1}}$, the reparametization of the hyperplane along the geodesic path along the hemisphere $S^{\lambda_G(\iota_{1_{\min}},\iota_{1_{\max}})}_{\omega_G(\iota_{1_{\min}},\iota_{1_{\max}})}$ will always be contained in $\Cleav_{S^n}(-;2)$. 

However, we need to consider these paths for our trees decorated by multiple hyperplanes. That we in \ref{enfunctorconstruction} are taking path-components tells us that the geodesic path along the hemisphere $S^{\lambda_G(\iota_{1_{\min}},\iota_{1_{\max}})}_{\omega_G(\iota_{1_{\min}},\iota_{1_{\max}})}$ parametrizing $P_{\underline{\iota}(1)}$ will be contained in the space associated to $\underline{\iota}$. Since the elements $[T,P_{\underline{\iota}}]$ and $[T',P_{\underline{\lambda}}]$ agree, the hyperplane $P_{\underline{\lambda}^{-1}(\underline{\iota}(1)}$ as a decoration of $[T',\underline{P}_{\underline{\lambda}}]$ will be bound to dominate the same hyperplanes as $P_{\underline{\iota}(1)}$ as an element of $[T,\underline{P}_{\underline{\iota}}]$, and this will remain true along the reparametization along the geodesic, until $P_{\underline{\iota}(1)}$ no longer dominates any other hyperplanes, since otherwise either $[T,\underline{P}_{\underline{\iota}}]$ or $[T,\underline{P}_{\underline{\lambda}}]$ would be reparametrized to non-cleaving trees. To ensure that no new dominance occurs after the potential step where $P_{\underline{\iota}(1)}$ no longer dominates any other hyperplanes, and the cleaving elements hence would disagree, one translates along the hyperplane using the $\RR$-factor determining its position in $\RR^{n+1}$, moving it at the same speed in the opposite direction as $P_{\underline{\iota}(1)}$ is moving away from the hyperplanes it used to dominate.

This eventually brings the hyperplanes $P_{\underline{\iota}(1)}$ parallel with the first coordinate axis of $\RR^{n+1}$, and one iterates this construction for the remaining $P_{\underline{\iota}(2)},\ldots,P_{\underline{\iota}(k-1)}$ in that order, to -- along a geodesic in the parametrizing hemispheres -- bring them all parallel to the first coordinate axis. Finally, having used step 2, one applies the inverse homotopy of the one from $a_1$ to $a_0$, to make these hyperplanes cleave $a_0$. 

As noted previously, step 1 now applies to finish the proof.
\end{proof}





\begin{theorem}\label{entheorem}
$\Cleav_{S^{n}}$ is a coloured $E_{n+1}$-operad.
\end{theorem}
\begin{proof}
The lemmas \ref{cofiberlemma}, \ref{limitlemma}, \ref{contractlemma} check (A)-(C) in \ref{cellulite} the functor $D$ of \ref{enfunctorconstruction} should satisfy in order for $\Cleav_{S^{n}}$ to be $E_{n+1}$. Note that \ref{contractlemma} indeed gives the full naturality as stated in (B), since we give an explicit homotopy equivalence onto $\Ob(\Cleav_{S^n}) \subset \Cleav_{S^n}(-;k)$ that respects morphisms of $S(\Ob(\Cleav_{S^n}))$.
\end{proof}

\begin{corollary}
There is an equivalence of operads 
$$\h_*(\Cleav_{S^n}) \cong \h_*(\Disk_{n+1})$$
\end{corollary}
\begin{proof}
By \ref{contractibletimber} the colours of $\Cleav_{S^n}$ are contractible, and by further \ref{evincfibration}, we can apply \ref{colourtomono}, so $\H_*(\Cleav_{S^n})$ is a monochrome operad and \ref{entheorem} together with \ref{knowen} gives the corollary.
\end{proof}

\section{Deforming Cleavages Symmetrically}\label{defsym}
Given a topological space, $X$, we let $\Homeo(X)$ denote the group of self-homeo\-mor\-phisms of $X$. 

Similar to having a monochrome $G$-operad in the category of $G$-spaces, for coloured topological operads we give the following definition: 

\begin{definition}\label{groupaction}
A topological group $G$ \emph{acts} on a coloured topological operad, $\mathpzc{O}$ if we are given
\begin{itemize}
\item $. \colon G \to \Homeo(\Ob(\mathpzc{O}))$ a continuous map. Given $g \in G$ and $U \in \Ob(\mathpzc{O})$, under adjunction, we denote the corresponding acted upon element as $g.U$.
\item $\alpha_i \colon G \to \Homeo(\mathpzc{O}(-;m))$ continous maps for all $i\in \{1,\ldots,m\}$ and $m \in \NN$.
\end{itemize}
and these respect the topological structure of operads; i.e. letting $g \in G$ the following diagram should be commutative for all $i \in \{1,\ldots,k\}$ and $j \in \{1,\ldots,m\}$:
\begin{eqnarray}\label{actiondiagram}
\xymatrix{
& & \mathpzc{O}(A;k+m-1)\ar[ddl]|(.6){\alpha_{i+j-1}(g)} \\
& & \mathpzc{O}(A;k) \times_{\Ob(\mathpzc{O})} \mathpzc{O}(-;m)\ar[r] \ar[u]_-{\circ_i}\ar[d]\ar[ddl] & \mathpzc{O}(-;m)\ar[d]|{\ev_{\inc}}\ar[ddl]|(.7){\alpha_j(g)}\\
& \mathpzc{O}(A;k+m-1) & \mathpzc{O}(A;k)\ar[r]^{\ev_{i}}\ar[r]\ar[ddl]|(.3){\alpha_i(g)} & \Ob(\mathpzc{O})\ar[ddl]|{.g}\\
 & \mathpzc{O}(A;k) \times_{\Ob(\mathpzc{O})} \mathpzc{O}(-;m)\ar[r]\ar[u]_-{\circ_i}\ar[d] & \mathpzc{O}(-;m)\ar[d]|{\ev_{\inc}}
\\
& \mathpzc{O}(A;k)\ar[r]^{\ev_{i}}\ar[r] & \Ob(\mathpzc{O}) & .
}
\end{eqnarray}
\end{definition}
In the classical setting, when $\mathpzc{O}$ is a monochrome operad, this recovers the notion of a $G$-operad. In this spirit we call a coloured operad satisfying the above a \emph{coloured G-operad}. 

In \cite[2.1]{WahlBat}, semidirect products of monochrome operads are introduced, and we can expand the notion to the coloured setting by only expanding a little on the operadic evaluation maps:

\begin{definition}\label{semidirect}
For a coloured topological operad $\mathpzc{O}$ with an action of a group $G$, we can form the \emph{semidirect product} of $\mathpzc{O}$ by $G$, as for the monochrome setting denoted $\mathpzc{O} \rtimes G$, by letting
\begin{itemize}
\item $\Ob(\mathpzc{O} \rtimes G) = \Ob(\mathpzc{O})$
\item $\mathpzc{O} \rtimes G = \mathpzc{O}(-;k) \times G^k$
\end{itemize}
Letting $\ev_i$ and $\ev_{\inc}$ denote the operadic evaluation maps of $\mathpzc{O}$. The operadic evaluation maps $\ev_i^G$ and $\ev_{\inc}^G$ for the semidirect product are given by
$\ev_i^G(\omega,\rho_1,\ldots,\rho_k) = \rho_i.\ev_i(\omega)$ and $\ev_{\inc}^G = \ev_{\inc}$. As for the monochrome case, the operadic composition $\circ_i \colon (\mathpzc{O} \rtimes G)(-;k) \times_{\Ob(\mathpzc{O})} (\mathpzc{O} \rtimes G)(-;m) \to (\mathpzc{O} \rtimes G)(-;k+m-1)$ is given by twisting the composition of $\mathpzc{O}$ through the action of $G$ in the following sense:

$$
(\omega;\rho_1,\ldots,\rho_k) \circ_i (\omega';\eta_1,\ldots,\eta_m) = (\omega \circ_i \rho_i.\omega';\rho_1,\ldots,\rho_{i-1},\rho_i.\eta_1,\ldots,\rho_i.\eta_m,\rho_{i+1},\ldots,\rho_k)
$$

\end{definition}

\begin{observation}\label{sonaction}
As a gadget constructed from $\RR^{n+1}$, there is an induced action of $\SO(n+1)$ on $\Cleav_{S^n}$, precisely:

To an affine oriented hyperplane $P \in \Hyp^{n+1}$ and $\rho \in \SO(n+1)$, letting $\rho$ act on $\RR^{n+1}$ by rotation leads $P$ to the affine oriented hyperplane $\rho.P$. Hereby $\SO(n+1)$ acts on the space $\Hyp^{n+1}$.

There is also an action of $\SO(n+1)$ on $\Ob(\Cleav_{S^n})$ rotating timber $U \in \Ob(\Cleav_{S^n})$ obtained by cleaving $S^n$ along some hyperplanes to the timber $\rho.U$ obtained by cleaving $S^n$ along the hyperplanes rotated along $\rho$. 

Let $[T,\underline{P}] \in \Cleav_{S^n}(U;k)$, for some $U \in \Ob(\Cleav_{S^n})$. Let $\rho.[T,\underline{P}] = [T,\rho.\underline{P}]$, understood in the sense that $\rho$ rotates the decorations of $T$ simultaneously. Having $[T,\underline{P}]$ an element of $\Cleav_{S^n}(U;k)$ this should be interpreted in the sense that $\rho.[T,\underline{P}]$ is an element of $\Cleav_{S^n}(\rho.U;k)$, ensuring that the decorated tree cleaving $\rho.U$ and in turn making this action of $\SO(n+1)$ on $\Cleav_{S^n}$ satisfy \ref{groupaction}.
\end{observation}

The action of \ref{sonaction} defines the semidirect product of $\Cleav_{S^n}$ by $\SO(n+1)$, in the sense of \ref{semidirect}. 

\begin{observation}\label{sonextension}
We can extend the action of $\Cleav_{S^n}$ along correspondances as given in \ref{cleavageaction} to an action of $\Cleav_{S^n} \rtimes {\SO(n+1)}$ on $M^{S^n}$. This is done for $([T,\underline{P}],\rho_1,\ldots,\rho_k) \in (\Cleav_{S^n} \rtimes \SO(n+1))(-;k)$ by the following pullback-diagram:
$$
\xymatrix{
M_{[T,\underline{P}]}^{S^n} \ar[r]^{\phi_{\underline{\rho}}^*}\ar[d] & \left(M^{S^n}\right)^k\ar[d]|{\res_{\underline{\rho}}} \\
 M^{\widehat{\pi_0(\beta_{[T,\underline{P}]})}} \ar[r]^{\phi} & \prod_{i=1}^k M^{\left(\complement N_{i}\right),}
}
$$ 
where all spaces are as in \ref{cleavageaction}, as is the map $\phi$. However, the \emph{twisted restriction map} $\res_{\underline{\rho}}$ is given at the $i$'th factor of $\left(M^N\right)^k$ as $\res_{\complement N_i}.\rho_i^{-1}$. Here the map $\res_X \colon M^{S^n} \to M^X$ denotes the restriction map to the space $X \subseteq S^n$. The element $\rho_i \in \SO(n+1)$ is considered as a diffeomorphism of $S^n$, and $\res_X.\rho_i^{-1}$ denotes the precomposition of $\rho_i^{-1}$ of the domain of $f \in M^{S^n}$ prior to applying the restriction map. This preapplication of $\rho_i^{-1}$ allows us to consider $\res_{\complement N_i}.\rho_i^{-1}$ as a map that takes points of $\rho_i(\complement N_i) \subseteq S^n$ and brings these to $\complement N_i$ where a restriction map is subsequently applied. 

Note that whereas this allows us to let $\phi$ be given as the same morphism as in \ref{cleavageaction}, which in turn allows us to again identify the pullback space $M^{S^n}_{[T,\underline{P}]}$ as maps from $S^n$ to $M$ that are constant along the blueprint of $[T,\underline{P}]$. However, the fact that we are applying a twisted restriction map means that the associated map $\phi_{\underline{\rho}}^*$ in the pullback will be different from \ref{cleavageaction}. Concretely, $\phi^*_{\underline{\rho}}$ is given by having the $i$'th image maps from $S^n$ to $M$ that are constant along $\rho_i(\complement N_i)$. 

Note that this description of $\phi^*_{\underline{\rho}}$ makes this action of correspondanes respect the operadic composition, in the sense that for $([T,\underline{P}],\rho_1,\ldots,\rho_k)$ and $([T',\underline{P}'],\eta_1,\ldots,\eta_m)$ $\circ_i$-composable as elements of $\Cleav_{S^n} \rtimes \SO(n+1)$, the change of colours by $\rho_i \colon S^n \to S^n$ in the operadic composition of the semidirect product makes it necessary for commutativity of operadic associatativity diagrams to map $M^{S^n}_{[T,\underline{P}] \circ_i \rho_i.[T',\underline{P}']}$ to the $i+j-1$'th factor of $\left(M^{S^n}\right)^{k+m-1}$ that are constant along the subspace $\eta_j(\rho_i(\complement N_{i+j-1})) \subseteq S^n$, where $j \in \{1,\ldots,m\}$ and $i \in \{1,\ldots,k\}$. 

\end{observation}

\begin{proposition}\label{CleaveBatalin}
Actions of $\h_*(\Cleav_{S^n} \rtimes {\SO(n+1}))$ are the same as actions of the homology of the framed little disks operad, in the sense of \cite[5.1]{WahlBat}.
\end{proposition}
\begin{proof}
Since $\Cleav_{S^n}$ is a coloured $E_{n+1}$-operad; \ref{entheorem}, with contractible colours as; \ref{contractibletimber}, we have that $\h_*(\Cleav_{S^n})$ is in particular a quadratic operad, with $3$-ary operations relations determined by the Gerstenhaber relations.

Hereby, the statement follows directly by applying \cite[4.4]{WahlBat}, to see that operadic actions of $\h_*(\Cleav_{S^n} \rtimes \SO(n+1))$ agrees with operadic actions of $\h_*(\Disk_{n+1} \rtimes \SO(n+1))$.

\end{proof}

It is natural to conjecture that there is a weak equivalence of operads $\Cleav_{S^n} \rtimes \SO(n+1) \simeq \Disk_{n+1} \rtimes \SO(n+1)$. However the string of equivalences of \ref{cellulite} does not produce an equivalence extending \ref{entheorem} to the setting of $\BV_{n+1}$-operads. One notices that the intermediate terms given in \ref{cellulite} involves the nerve of the full graph operad. Heurestically, this nerve construction does not see the action of $\SO(n+1)$ on $\RR^{n+1}$. One option for proving such an extension of \ref{entheorem} could be to attempt at making $\SO(n+1)$-equivariant versions of the nerve of the full graph operad. We shall however not go into these considerations in this paper.

\pagebreak
\pagebreak
\bibliography{referencer}
\end{document}